\documentclass[11pt,epsf]{article}
\usepackage{graphicx}
\usepackage{amsthm}
\usepackage{amsfonts}
\usepackage{amssymb}
\title{On higher-order discriminants}
\author{Vladimir Petrov Kostov\\ Universit\'e C\^ote d'Azur, CNRS, LJAD, France
\\  
e-mail: kostov@math.unice.fr} 
\date{}
\newtheorem{tm}{Theorem}
\newtheorem{defi}[tm]{Definition}
\newtheorem{rem}[tm]{Remark}
\newtheorem{rems}[tm]{Remarks}
\newtheorem{lm}[tm]{Lemma}
\newtheorem{ex}[tm]{Example}

\newtheorem{prop}[tm]{Proposition}
\newtheorem{nota}[tm]{Notation}
\newtheorem{st}[tm]{Statement}

\newtheorem{obs}[tm]{Observation}
\begin{document} 
\maketitle 

\begin{abstract}
For the family of polynomials in one variable 
$P:=x^n+a_1x^{n-1}+\cdots +a_n$, $n\geq 4$, we consider its 
higher-order discriminant sets $\{ \tilde{D}_m=0\}$, where 
$\tilde{D}_m:=$Res$(P,P^{(m)})$, 
$m=2$, $\ldots$, $n-2$, and their projections in the spaces of 
the variables $a^k:=(a_1,\ldots ,a_{k-1},a_{k+1},\ldots ,a_n)$. 
Set $P^{(m)}:=\sum _{j=0}^{n-m}c_ja_jx^{n-m-j}$, $P_{m,k}:=c_kP-x^mP^{(m)}$. 
We show that Res$(\tilde{D}_m,\partial \tilde{D}_m/\partial a_k,a_k)=
A_{m,k}B_{m,k}C_{m,k}^2$, where $A_{m,k}=a_n^{n-m-k}$, 
$B_{m,k}=$Res$(P_{m,k},P_{m,k}')$ 
if $1\leq k\leq n-m$ and $A_{m,k}=a_{n-m}^{n-k}$, 
$B_{m,k}=$Res$(P^{(m)},P^{(m+1)})$ if $n-m+1\leq k\leq n$.
The equation $C_{m,k}=0$ defines the projection in the space 
of the variables $a^k$ 
of the closure of the set of values of $(a_1,\ldots ,a_n)$ 
for which $P$ and $P^{(m)}$ have two 
distinct roots in common. The polynomials $B_{m,k},C_{m,k}\in \mathbb{C}[a^k]$ 
are irreducible. The result is generalized to the case when $P^{(m)}$ 
is replaced by a polynomial $P_*:=\sum _{j=0}^{n-m}b_ja_jx^{n-m-j}$, 
$0\neq b_i\neq b_j\neq 0$ for $i\neq j$.\\ 


\end{abstract}

\section{Introduction}

In this paper we consider for $n\geq 4$ the general family of monic polynomials 
in one variable 
$P(x,a):=x^n+a_1x^{n-1}+\cdots +a_n$, $x,a_j\in \mathbb{C}$. For its $m$th 
derivative w.r.t. $x$ we set 
$P^{(m)}:=c_0x^{n-m}+c_1a_1x^{n-m-1}+\cdots +c_{n-m}a_{n-m}$, 
where $c_j=(n-j)!/(n-m-j)!$. 
For $m=1$, $\ldots$, 
$n-1$ we define the {\em $m$th order discriminant} of $P$ as 
$\tilde{D}_m:=$Res$(P,P^{(m)})$ which is the determinant of the Sylvester matrix 
$S(P, P^{(m)})$. We remind that $S(P, P^{(m)})$ is 
$(2n-m)\times (2n-m)$, its first (resp. $(n-m+1)$st) row equals 

$$(1,a_1,\ldots ,a_n,0,\ldots ,0)~~\, \, \, {\rm (resp.}~~\, \, \, 
(c_0,c_1a_1,\ldots ,c_{n-m}a_{n-m},0,\ldots ,0)~~\, {\rm )~~,}$$
the second (resp. $(n-m+2)$nd) row is obtained from this one by shifting 
by one position to the right and by adding $0$ to the left etc. We say that 
the variable $a_j$ is of {\em quasi-homogeneous weight} $j$ because up to 
a sign it equals the $j$th elementary symmetric polynomial in the roots of 
the polynomial $P$; the quasi-homogeneous weight of $x$ is~$1$.  

There are at least two problems in which such discriminants are of interest. 
One of 
them is the Casas-Alvero conjecture that if a complex univariate polynomial has 
a root in common with each of its nonconstant derivatives, then it is a power 
of a linear polynomial, see \cite{CLO}, \cite{Y1} and \cite{Y2} 
and the claim in \cite{Y} that the answer to 
the conjecture is positive. 

Another one is the study of the 
possible arrangements of the roots of a hyperbolic polynomial (i.e. real and 
with all roots real) and of all its nonconstant derivatives on the real line. 
This problem can be generalized to a class of polynomial-like functions 
characterized by the property their $n$th derivative to vanish nowhere. 
It turns out that for this class Rolle's theorem gives only necessary, 
but not sufficient 
conditions for realizability of a given arrangement by the zeros of 
a polynomial-like function, see \cite{Ko4}, \cite{Ko5}, \cite{Ko6} 
and \cite{Ko7}. Pictures 
of discriminants for the cases $n=4$ and $n=5$ can be found in \cite{Ko1}. 
Properties of the discriminant set $\{ \tilde{D}_1=0\}$ for real polynomials 
are proved in \cite{Me}. 

A closely related question to the one of the 
arrangement of the roots of a hyperbolic polynomial is the one 
to study {\em overdetermined strata} in the space of the coefficients of the 
family of polynomials $P$ (the definition is given by B.~Z.~Shapiro 
in \cite{KoSh}); these are sets of values of the coefficients for which there 
are more equalities between roots of the polynomial and its derivatives than 
expected. Example: the family of polynomials $x^4+ax^3+bx^2+cx+d$ depends on 
$4$ parameters two of which can be eliminated by shifting and rescaling the 
variable $x$ which gives (up to a nonzero constant factor) the family 
$S:=x^4-x^2+cx+d$. For $c=0$, $d=1/2$ the polynomial has two double roots 
$\pm 1/\sqrt{2}$, and $0$ is a common root for $S'$ and $S'''$. This makes 
three independent equalities, i.e. more than the number of parameters. For 
polynomials of small degree, overdetermined strata have been studied in 
\cite{EHHS} and \cite{EK}. The study of overdetermined strata is interesting 
both in the case of complex and in the case of real coefficients. 

In what follows we enlarge the context by considering instead of the 
couple of polynomials $(P,P^{(m)})$ the couple $(P,P_*)$, where 
$P_*:=\sum _{j=0}^{n-m}b_ja_jx^{n-m-j}$, $b_j\neq 0$ and $b_i\neq b_j$ 
for $i\neq j$. By abuse of notation we set $\tilde{D}_m:=$Res$(P,P_*)$.

\begin{prop}\label{propmonom}
The polynomial $\tilde{D}_m$ is irreducible. It is a degree $n$ polynomial 
in each of the variables $a_j$, $j=1$, $\ldots$, $n-m$, and a degree 
$n-m$ polynomial in each of the variables $a_j$, $j=n-m+1$, $\ldots$, $n$. 
It contains monomials $M_j:=\pm b_j^na_j^n(1-b_0/b_j)^ja_n^{n-m-j}$, 
$j=1$, $\ldots$, $n-m$, and 
$N_s:=\pm b_{n-m}^{m-s}a_{n-m}^{m-s}b_0^{n-m+s}a_{n-m+s}^{n-m}$, $s=1$, $\ldots$, $m-1$. 
It is quasi-homogeneous, of 
quasi-homogeneous weight $n(n-m)$. The monomial $M_j$ (resp. $N_s$) is 
the only monomial containing $a_j^n$ (resp. $a_{n-m+s}^{n-m}$).
\end{prop}

\begin{proof}
We prove first the presence in $\tilde{D}_m$ of the monomials $M_j$ and $N_s$. 
For each $j$ fixed, $1\leq j\leq n-m$, one can subtract the $(n-m+\nu )$th row 
of $S(P,P_*)$ multiplied by $1/b_j$ from its $\nu$th one, 
$\nu =1$, $\ldots$, $n-m$. We denote by $T$ the new matrix.
One has $\det T=\det S(P,P_*)$ and  
the variable $a_j$ is not present in the first $n-m$ rows of $T$. Thus there 
remains a single term of $\det T$ containing $n$ factors $a_j$; it is 
obtained when the entries $b_ja_j$ in positions $(n-m+\mu ,j+\mu )$ of $T$, 
$\mu =1$, $\ldots$, $n$, are multiplied by the entries $a_n$ in positions 
$(\ell ,n+\ell )$, $\ell =j+1$, $\ldots$, $n-m$, 
and by the entries $1-b_0/b_j$ in 
positions $(\ell ,\ell )$, $\ell =1$, $\ldots$, $j$; this gives the monomial 
$M_j$. (If when computing $\det S(P,P_*)$ 
one chooses to multiply the $n$ entries $b_ja_j$, then they must 
be multiplied by entries of the matrix obtained from $S(P,P_*)$ by deleting 
the rows and columns of the entries $b_ja_j$. This matrix is block-diagonal, 
its upper left block is upper-triangular, with diagonal entries equal to 
$1-b_0/b_j$, its right lower block is lower-diagonal, with diagonal 
entries equal to $a_n$. Hence $M_j$ is the only monomial containing $n$ 
factors $a_j$.)

To obtain the monomial $N_s$ one chooses in the definition of $T$ 
above $j=n-m$. Hence the first $n-m$ 
rows of $T$ do not contain the variable $a_{n-m}$. The monomial $N_s$ is 
obtained by multiplying the entries $a_{n-m+s}$ in positions 
$(r,n-m+s+r)$, $r=1$, $\ldots$, $n-m$, by the entries $b_{n-m}a_{n-m}$ in 
positions $(q,q)$, $q=2n-2m+s+1$, $\ldots$, $2n-m$ and by the entries $b_0$ 
in positions $(n-m+p,p)$, $p=1$, $\ldots$, $n-m+s$. The monomial $N_s$ 
is the only one containing $n-m$ factors $a_{n-m+s}$ (proved by analogy with 
the similar claim about the monomial $M_j$).

The matrix $S(P,P_*)$ contains each of the variables $a_j$, $j=1$, $\ldots$, 
$n-m$ (resp. $a_s$, $s=n-m+1$, $\ldots$, $n$) 
in exactly $n$ (resp. $n-m$) of its columns. 
The presence of the monomials $M_j$ (resp. $N_s$) in $\tilde{D}_m$ shows that 
$\tilde{D}_m$ is a degree $n$ polynomial in the variables $a_j$ and a degree 
$n-m$ one in the variables $a_s$. 

Quasi-homogeneity of $\tilde{D}_m$ follows from the fact that its zero set 
and the zero sets of the polynomials $P$ and $P_*$ remain 
invariant under the quasi-homogeneous dilatations $x\mapsto tx$, 
$a_{\kappa}\mapsto t^{\kappa}a_{\kappa}$, $\kappa =1$, $\ldots$, $n$. 
Each of the monomials $M_j$ and $N_s$ is of quasi-homogeneous weight $n(n-m)$. 

Irreducibility of $\tilde{D}_m$ results from the impossibility to present 
simultaneously all monomials $M_j$ and $N_s$  
as products of two monomials, 
of quasi-homogeneous weights $u$ and $n(n-m)-u$, 
for any $1\leq u\leq n(n-m)-1$. 
\end{proof}

\begin{nota}
{\rm For $Q,R\in \mathbb{C}[x]$ we denote by Res$(Q,R)$ the resultant of $Q$ 
and $R$ and we write $P^{(m)}$ for $d^mP/dx^m$. 
This refers also to the case when the coefficients of $Q$ and $R$ 
depend on parameters. We set $a:=(a_1,\ldots ,a_n)$ (resp. 
$a^j=(a_1,\ldots ,a_{j-1},a_{j+1},\ldots ,a_n)$) and we denote by 
$\mathcal{A}\simeq \mathbb{C}^n$ (resp. $\mathcal{A}^j\simeq \mathbb{C}^{n-1}$) 
the space of the variables $a$ (resp. $a^j$). 
For $K,L\in \mathbb{C}[a]$ we write $S(K,L,a_k)$ and Res$(K,L,a_k)$ 
for the Sylvester matrix and the resultant of $K$ and $L$ 
when considered as polynomials in $a_k$. We set 
$\tilde{D}_{m,k}:=$Res$(\tilde{D}_m,\partial \tilde{D}_m/\partial a_k,a_k)$. 
For a matrix $A$ we denote by $A_{k,\ell}$ its entry in position $(k,\ell )$ 
and by $[A]_{k,\ell }$ the matrix obtained from $A$ 
by deleting its $k$th row and $\ell$th column. By $\Omega$ 
(indexed, with accent or not) we denote throughout the paper 
nonspecified nonzero constants. 
By $P_{m,k}$ ($1\leq k\leq n-m$) we denote the polynomial $b_kP-x^mP_*$;  
its coefficients of $x^n$ and $x^k$ equal $b_k-b_0\neq 0$ and $0$.}
\end{nota} 

\begin{defi}\label{ThetaM}
{\rm For $1\leq m\leq n-2$ we denote by $\Theta$ and $\tilde{M}$ 
the subsets of the hypersurface 
$\{ \tilde{D}_m=0\} \subset \mathcal{A}$ such that for $a\in \Theta$ (resp. 
for $a\in \tilde{M}$) the polynomial $P$ has a root which is a double 
root of $P_*$ (resp. the polynomials $P$ and $P_*$ have two 
simple roots in common). The remaining roots of $P$ and $P_*$  
are presumed 
simple and mutually distinct. We call the set $\tilde{M}$ 
{\em the Maxwell stratum} of $\{ \tilde{D}_m=0\}$.}   
\end{defi}

In the present paper we prove the following theorem;

\begin{tm}\label{maintm}
Suppose that $2\leq m\leq n-2$. Then:

(1) The polynomial $\tilde{D}_{m,k}$ can be represented in the form 

\begin{equation}\label{formula}
\tilde{D}_{m,k}=A_{m,k}B_{m,k}C_{m,k}^2~,
\end{equation} 
where $A_{m,k}=a_n^{n-m-k}$ if $k=1$, $\ldots$, $n-m$, and 
$A_{m,k}=a_{n-m}^{n-k}$ if $k=n-m+1$, $\ldots$, $n$, 
$B_{m,k}$ and $C_{m,k}$ are irreducible 
polynomials in the variables $a^k$.

(2) One has $B_{m,k}=${\rm Res}$(P_{m,k},P_{m,k}')$ if $k=1$, $\ldots$, $n-m$, and 
$B_{m,k}=${\rm Res}$(P_*,P_*')$ if $k=n-m+1$, $\ldots$, $n$.

(3) The equation $C_{m,k}=0$ defines the projection in the space $\mathcal{A}^k$ 
of the closure of the Maxwell stratum.
\end{tm}

The paper is structured as follows. After some examples and remarks in 
Section~\ref{secexrem}, we justify in  
Section~\ref{Amksec} the form of the factor $A_{m,k}$, see 
Proposition~\ref{Amkprop}; Section~\ref{Amksec} begins with Lemma~\ref{lm2diag} 
which gives the form of the determinant of 
certain matrices that appear in the proof of Theorem~\ref{maintm}. 
Section~\ref{seclmst} contains  
Lemma~\ref{lmsmooth} and Statements~\ref{stMaxwell}, \ref{stTheta} 
and~\ref{stirred} (the latter claims that the factors $B_{m,k}$ and $C_{m,k}$ 
are irreducible). They  
imply that one has 
$\tilde{D}_{m,k}=A_{m,k}B_{m,k}^{s_{m,k}}C_{m,k}^{r_{m,k}}$, where $s_{m,k}$, $r_{m,k}$ 
$\in \mathbb{N}$, see 
Remark~\ref{remintermediate}. Thus after Section~\ref{seclmst} there 
remains to show only that 
$s_{m,k}=1$ and $r_{m,k}=2$. 
In Section~\ref{secmn-2} we prove Theorem~\ref{maintm} 
in the case $m=n-2$, see Proposition~\ref{propmn-2}. In Section~\ref{seck=1} 
we show that $s_{m,k}=1$. We finish the proof 
of Theorem~\ref{maintm} in Section~\ref{seccompletion}, by induction on $n$ 
and $m$, as follows. Statement~\ref{stA} deduces formula (\ref{formula}) 
for $n=n_0+1$, $k=k_0+1$ from formula (\ref{formula}) 
for $n=n_0$, $k=k_0$. Statement~\ref{stAA} justifies formula (\ref{formula}) 
for $n=n_0$, $2\leq m<n_0-2$, $k=1$ 
using formula (\ref{formula}) for $n=n_0$, $m=n_0-2$, $k=1$ 
(recall that the latter is justified in Section~\ref{secmn-2}).

{\bf Acknowledgement.} The author is deeply grateful to B.~Z. Shapiro from 
the University of Stockholm for having 
pointed out to him the importance to study discriminants 
and for the fruitful discussions of this subject.

\section{Examples and remarks\protect\label{secexrem}}

Although Theorem~\ref{maintm} speaks about the case $2\leq m\leq n-2$,  
our first example treats 
the case $m=1$ in order to show its differences with the case 
$2\leq m\leq n-2$:
  
\begin{ex}\label{exm=1}
{\rm For $n=3$, $m=1$ we set $P:=x^3+ax^2+bx+c$, $P_*:=x^2+Aax+Bb$, 
$0\neq A,B\neq 1$, $A\neq B$. Then} 

$$\begin{array}{lcl}
\tilde{D}_1&=&(1-A)B(B-A)a^2b^2+(3AB-A-2B)abc+c^2+
A^2(1-A)a^3c+B(1-B)^2b^3\\ \\ 
\tilde{D}_{1,1}&=&-A^2(A-1)^2\, c\, (-27A^2(1-A)c^2+4(A-B)^3b^3)\, 
(-Ac^2+(1-A)B^2(1-B)b^3)^2\\ \\ 
\tilde{D}_{1,2}&=&
-B^2(B-1)^2\, c\, (-27B(1-B)^2c+4(A-B)^3a^3)\, (-(1-B)c+A(1-A)^2Ba^3)^2\\ \\  
\tilde{D}_{1,3}&=&-(-4Bb+A^2a^2)\, ((1-B)b-A(1-A)a^2)^2~.\end{array}$$
{\rm The condition $P$ and $P_*$ to have two roots in common is 
tantamount to $P_*$ dividing $P$. One has $P=(x+a(1-A))P_*+W_1x+W_0$, where} 
$$W_1:=(1-B)b-A(1-A)a^2~~,~~W_0:=c-B(1-A)ab~.$$
{\rm The quadratic factors in the above presentations of $\tilde{D}_{1,k}$, 
$k=1$, $2$ and $3$, are obtained by eliminating respectively $a$, $b$ and $c$ 
from the system of equations $W_1=W_0=0$ which is the necessary and sufficient 
condition $P_*$ to divide $P$.

In the particular case $A=2/3$, $B=1/3$ (i.e. $P_*=P'/3$) one obtains}
$$\tilde{D}_{1,1}=(-2^6/3^{15})\, c\, (-27c^2+b^3)^3~~,~~
\tilde{D}_{1,2}=(-2^6/3^{15})\, c\, (-27c+a^3)^3~~,~~
\tilde{D}_{1,3}=(2^4/3^6)\, (3b-a^2)^3~.$$

\end{ex}

\begin{rems}\label{remcompare}
{\rm (1) For $n\geq 4$, $m=1$ and $P_*=P'$ 
a result similar to Theorem~\ref{maintm} holds 
true. Namely, if $n\geq 4$, then $\tilde{D}_{1,k}$ 
is of the form $A_{1,k}B_{1,k}^3C_{1,k}^2$, where for $m=1$ the polynomials 
$B_{m,k}$ and $C_{m,k}$ are defined in the same way as for $2\leq m\leq n-2$ 
(with $P_*=P'$), but  
$A_{1,k}=a_n^{\min (1,n-k)+\max (0,n-k-2)}$, see \cite{Ko2} and \cite{Ko3}. 
Hence for $m=1$ and $P_*=P^{(m)}$ 
there are two differences w.r.t. the case $m\geq 2$ -- the degree $3$ 
(instead of $1$) of 
$B_{1,k}$, and $A_{1,n-1}=a_n$ (instead of $A_{1,n-1}=1$). This difference can be 
assumed to stem 
from the fact that for $m=1$, if $P$ has a root of multiplicity $\geq 3$, then 
this is a root of multiplicity $\geq 2$ for $P'$. This explanation is 
detailed below and in Remark~\ref{remm=1}. 

For $n=4$ and for generic values of $b_j$ the polynomials 
$\tilde{D}_{1,k}$, up to a constant nonzero factor, are of the form} 

$$\begin{array}{ccrcccrc}
\tilde{D}_{1,1}&=&(b_1/b_0)^3(1-b_1/b_0)^2\, a_4^2\, 
\tilde{B}_{1,1}\tilde{C}_{1,1}^2&,&\tilde{D}_{1,2}&=&
-(b_2/b_0)^2(1-b_2/b_0)^2\, a_4\, \tilde{B}_{1,2}\tilde{C}_{1,2}^2&,\\ \\  
\tilde{D}_{1,3}&=&
-(b_3/b_0)^2(1-b_3/b_0)^3\, a_4\, \tilde{B}_{1,3}\tilde{C}_{1,3}^2&,&
\tilde{D}_{1,4}&=&\tilde{B}_{1,4}\tilde{C}_{1,4}^2&,\end{array}$$
{\rm where the polynomials $\tilde{B}_{1,k}$ and $\tilde{C}_{1,k}$, when 
considered as polynomials in the variables $a_j$ and $b_j$, are irreducible. 
Set $b_1=3b_0/4$, $b_2=b_0/2$, $b_3=b_0/4$. This is the case $P_*=P'$; we 
write $\tilde{B}_{1,k}|_{b_1=3b_0/4, b_2=b_0/2, b_3=b_0/4}=B_{1,k}$ and 
$\tilde{C}_{1,k}|_{b_1=3b_0/4, b_2=b_0/2, b_3=b_0/4}=C_{1,k}$. In this case 
the polynomials $\tilde{C}_{1,k}$ become reducible; they equal 
$B_{1,k}C_{1,k}$ which explains the presence of the cubic factor $B_{1,k}^3$.

Thus for $m=1$ the genericity condition 
$0\neq b_j\neq b_i\neq 0$ (which we assume to hold true 
in the formulation of Theorem~\ref{maintm}) is not sufficient in order 
to have the presentation (\ref{formula}) for $\tilde{D}_{m,k}$. At the same 
time imposing a more restrictive condition means leaving outside the 
most interesting case $P_*=P'$.}

{\rm (2) For $m=n-1$ 
the analog of the factor $C_{m,k}$ does not exist because $P_*$ has a single 
root $-b_1/b_0$. For $P_*=P^{(n-1)}:=n!(x+a_1/n)$ this is $x=-a_1/n$. 
In this case one finds that 
$\tilde{D}_{n-1}=(-1)^n(n!)^nP(-a_1/n)$. 
To see this one subtracts 
for $j=1$, $\ldots$, $n$ the $j$th column of the Sylvester matrix 
$S(P,x+a_1/n)$ multiplied by $-a_1/n$ from its $(j+1)$st column. This yields 
an $(n+1)\times (n+1)$-matrix $W$ whose entry in position $(1,n+1)$ equals 
$P(-a_1/n)$ and which below the first row has units in positions 
$(\nu +1,\nu )$, $\nu =1$, $\ldots$, $n$, and zeros elsewhere. 
Hence $\det W=(-1)^nP(-a_1/n)$. There remains to remind that 
$\tilde{D}_{n-1}=\det S(P,n!(x+a_1/n))=(n!)^n\det W$.

One finds directly that 
$\tilde{D}_{n-1,k}=\partial \tilde{D}_{n-1}/\partial a_k=
(-1)^n(n!)^n(-a_1/n)^{n-k}$, 
$2\leq k\leq n$. To find also $\tilde{D}_{n-1,1}$ one first observes that 
$$P_{n-1,1}(x)/(n-1)!=-(n-1)x^n+a_2x^{n-2}+a_3x^{n-3}+\cdots +a_n$$ 
and that 
$P(-a_1/n)=P_{n-1,1}(-a_1/n)/(n-1)!$. Hence up to a nonzero rational factor the 
determinants of the matrices $S(P_{n-1,1},P_{n-1,1}')$ and 
$S(\tilde{D}_{n-1},\partial \tilde{D}_{n-1}/\partial a_1,a_1)$ coincide, i.e. 
$\tilde{D}_{n-1,1}=\hat{c}\, $Res$(P_{n-1,1},P_{n-1,1}')$, 
$\hat{c}\in \mathbb{Q}$.

(3) The fact that the factor $C_{m,k}$ is squared (see formula (\ref{formula})) 
is not astonishing. At a generic point of the Maxwell stratum the hypersurface 
$\{ \tilde{D}_m=0\} \subset \mathcal{A}$ is locally the intersection 
of two analytic hypersurfaces, see Statement~\ref{stMaxwell}. Consider a point 
$\Psi \in \mathcal{A}^k$ close to the projection $\Lambda _0$ 
in $\mathcal{A}^k$ of a generic 
point $\Lambda \in \tilde{M}$. There exist two points 
$K_j\in \{ \tilde{D}_m=0\}$, 
$j=1$, $2$, which belong to these hypersurfaces and are close to $\Lambda$, and 
whose common projection in $\mathcal{A}^k$ 
is $\Psi$.  
There exists 
a loop $\gamma \subset \mathcal{A}^k$, $\Psi \in \gamma$, 
which  circumvents the projection in $\mathcal{A}^k$ 
of the set $\overline{\Theta \cup \tilde{M}}$ such that if one follows 
the two liftings on $\{ \tilde{D}_m=0\}$ of the points of $\gamma$ which 
at $\Psi$ are the points $K_j$, then upon one tour along $\gamma$ these 
liftings are exchanged. Hence in order to define the projection of $\tilde{M}$ in 
$\mathcal{A}^k$ by the zeros of an analytic function one has to 
eliminate this monodromy of rank $2$ by taking the square of $C_{m,k}$. For 
the case $m=1$ a detailed construction of such a path $\gamma$ is given 
in~\cite{Ko3}.}
\end{rems}

\begin{figure}[htbp]
\centerline{\hbox{\includegraphics[scale=0.7]{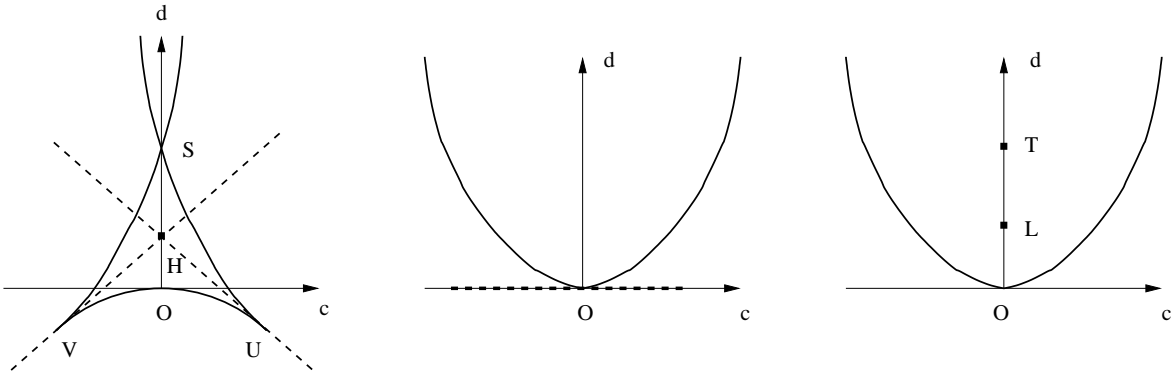}}}
    \caption{The sets $\{ \tilde{D}_2=0\} |_{a=0,b=-1}$, 
$\{ \tilde{D}_2=0\} |_{a=b=0}$ 
and $\{ \tilde{D}_2=0\} |_{a=0,b=1}$ for $n=4$.}
\label{ss}
\end{figure}

\begin{ex}
{\rm For $n=4$ we consider the case of real polynomials. 
We write $P=x^4+ax^3+bx^2+cx+d$ and we limit ourselves to 
the situation when $P_*:=P^{(m)}$. 
On Fig.~\ref{ss} we show the sets $\{ \tilde{D}_1=0\} |_{a=0}$ 
and $\{ \tilde{D}_2=0\} |_{a=0}$ when $b$, $c$ and $d$ are real. 
The sets $\{ \tilde{D}_1=0\} |_{a=0}$ and $\{ \tilde{D}_2=0\} |_{a=0}$ 
are invariant under the 
quasi-homogeneous dilatations $a\mapsto ta$, $b\mapsto t^2b$, $c\mapsto t^3c$, 
$d\mapsto t^4d$, therefore the intersections of the sets with the subspaces 
$\{ b=0\}$ and $\{ b=\pm 1\}$ give a sufficient idea about them. For each of 
these three intersections we represent the axes $c$ and $d$, see Fig.~\ref{ss}. 
For $b=-1$ the set $\{ \tilde{D}_1=0\} |_{a=0}$ is a curve with one 
self-intersection point at $S$ and two ordinary $2/3$-cusps at $U$ and $V$; 
it is drawn in solid line. At $U$ and $V$ the polynomial $P$ 
has one triple and one simple real root. The set 
$\{ \tilde{D}_2=0\} |_{a=0,b=-1}$ consists of two 
straight (dashed) lines intersecting at $H$ and tangent to the set 
$\{ \tilde{D}_1=0\} |_{a=0}$ at the cusps $U$ and $V$.

The sets $\{ \tilde{D}_1=0\} |_{a=0,b=0}$ and $\{ \tilde{D}_1=0\} |_{a=0,b=1}$ are 
parabola-like curves, the former has a $4/3$-singularity at the origin while 
the latter is smooth everywhere. 
The set $\{ \tilde{D}_1=0\} |_{a=0,b=1}$ contains 
an isolated double point $T$. 
The set $\{ \tilde{D}_2=0\} |_{a=0,b=0}$ (resp. 
$\{ \tilde{D}_2=0\} |_{a=0,b=1}$) is the $c$-axis (resp. the point $L$). 
The points $S$, $T$ and for $b=0$ the origin 
belong to a parabola (because the quasi-homogeneous weights of the 
variables $a_2$ and $a_4$ equal $2$ and $4$ respectively). 
So do the points $H$, $L$ 
and the origin for $b=0$. At $S$ (resp. $T$) the polynomial $P$ has two 
real (resp. two imaginary conjugate) double roots. At $H$ and $L$ the 
polynomial $P$ is divisible by $P''$.  

Globally the set $\{ \tilde{D}_2=0\} |_{a=0}$ is diffeomorphic 
to a Whitney umbrella. The set $\{ \tilde{D}_2=0\} |_{a=0}$ is smooth 
along the $c$-axis for $b=0$ (except at the origin) 
and its tangent plane is the $cd$-plane.}
\end{ex}

\section{The factor $A_{m,k}$\protect\label{Amksec}}

The following lemma will be used in several places of this paper:

\begin{lm}\label{lm2diag}
Consider a $p\times p$-matrix $A$ having nonzero entries only 

i) on the diagonal (denoted by $r_j$, 
in positions $(j,j)$, $j=1$, $\ldots$, $p$);  

ii) in positions 
$(\nu ,\nu +s)$, $\nu =1$, $\ldots$, $p-s$ (denoted by $q_{\nu}$), 
$1\leq s\leq p-1$, and 

iii) in positions $(\mu +p-s,\mu )$, $\mu =1$, $\ldots$, $s$ (denoted by 
$q_{\mu +p-s}$). 

Then $\det A=r_1\cdots r_p\pm q_1\cdots q_p$.
\end{lm}

\begin{proof}
Developing $\det A$ w.r.t. its first row one obtains the equality

$$\det A=r_1\det B+(-1)^{s+2}q_1\det C~~,~~{\rm where}~~\, \, 
B=[A]_{1,1}~~\, \, {\rm and}~~\, \, C=[A]_{1,s+1}~.$$
The matrix $B$ contains $p-1$ entries $r_j$ (namely, $r_2$, $\ldots$, $r_p$) 
and $p-2$ entries $q_{\nu}$ (the ones with $1\neq \nu \neq 1+p-s$). In the same 
way, the matrix $C$ contains $p-2$ entries $r_j$ ($1\neq j\neq s+1$) 
and $p-1$ entries $q_{\nu}$ ($\nu \neq 1$). 

When finding $\det B$ one can develop it w.r.t. 
that row or column in which there is an entry $r_j$ and there is no entry 
$q_{\nu}$. By doing so $p-1$ times one finds that 
$\det B=r_2\cdots r_p$. The $+$ sign of this product follows from 
the entries $r_j$ being situated on the diagonal. When finding $\det C$ one can 
develop it w.r.t. that row or column in which there is an entry $q_{\nu}$ 
and there is no entry $r_j$. By doing so $p-1$ times one finds that 
$\det C=\pm q_2\cdots q_p$ which proves the lemma.
\end{proof}

In the present section we prove the following proposition:

\begin{prop}\label{Amkprop}
(1) For $k=n-m+1$, $\ldots$, $n$, the polynomial $\tilde{D}_{m,k}$ is 
not divisible by any of the variables $a_j$, $j\neq n-m$.

(2) For $k=1,\ldots ,n-m$, the polynomial $\tilde{D}_{m,k}$ is 
not divisible by any of the variables $a_j$, $j\neq n$.

(3) For $k=1$, $\ldots$, $n-m$, the polynomial $\tilde{D}_{m,k}$ is divisible by 
$a_n^{n-m-k}$ and not divisible by $a_n^{n-m-k+1}$.

(4) For $k=n-m+1$, $\ldots$ ,$n$, it is divisible by 
$a_{n-m}^{n-k}$ and not divisible by $a_{n-m}^{n-k+1}$. 
\end{prop}


\begin{proof}[Proof of part (1):]
We show first that for $a_i=0$, $n-m\neq i\neq k$, the polynomial $\tilde{D}_m$ 
is of the form $\Omega 'a_{n-m}^n+\Omega ''a_{n-m}^{n-k}a_k^{n-m}$. 
Indeed, in this case one can list 
the nonzero entries of the 
$(2n-m)\times (2n-m)$-matrix $S(P,P_*)$ 
and the positions in which they are situated: 

$$\begin{array}{ccccccl}
1&(j,j)&,&a_{n-m}&(j,j+n-m)&,&\\ \\ 
a_k&(j,j+k)&&&&,&j=1,\ldots ,n-m~,\\ \\ 
b_0&(\nu +n-m,\nu )&,&b_{n-m}a_{n-m}&
(\nu +n-m,\nu +n-m)&,&
\nu =1,\ldots ,n~.\end{array}$$ 
Subtract the $(\mu +n-m)$th row multiplied by $1/b_{n-m}$ from the $\mu$th 
one for $\mu =1$, $\ldots$, $n-m$. This makes disappear the terms 
$a_{n-m}$ in 
positions $(j,j+n-m)$ while the terms $1$ in positions $(j,j)$ become 
equal to $\Omega _*:=1-b_0/b_{n-m}$. 
The determinant of the matrix doesn't change. We 
denote the new matrix by $T$. 
To compute $\det T$ one can develop it $n-k$ times 
w.r.t. the last column; each time one has a single nonzero entry 
in this column, this is $b_{n-m}a_{n-m}$ in  
position $(2n-m-\ell ,2n-m-\ell )$, $\ell =0$, $\ldots$, $n-k-1$.
The matrix $T_1$ which remains after deleting 
the last $n-k$ rows and columns of 
$T$ has the following nonzero entries, in the following positions:

$$\begin{array}{ccccccl}
\Omega _*&(j,j)&,&a_k&(j,j+k)&,&j=1,\ldots ,n-m~,\\ \\ 
b_0&(\nu +n-m,\nu )&,&b_{n-m}a_{n-m}&
(\nu +n-m,\nu +n-m)&,&
\nu =1,\ldots ,k~.\end{array}$$ 
Clearly $\det T=(b_{n-m}a_{n-m})^{n-k}\det T_1$. 
On the other hand the matrix $T_1$ satisfies the conditions of 
Lemma~\ref{lm2diag} with $p=n-m+k$ and $s=k$. Hence 
$\det T_1=\tilde{\Omega}'a_{n-m}^k+\tilde{\Omega}''a_k^{n-m}$ and  
$\tilde{D}_m=\Omega 'a_{n-m}^n+\Omega ''a_{n-m}^{n-k}a_k^{n-m}$.

But then the $(2n-2m-1)\times (2n-2m-1)$-Sylvester matrix 
$S^*:=S(\tilde{D}_m,\partial \tilde{D}_m/\partial a_k,a_k)$ 
has only the following 
nonzero entries, in the following positions:

$$\begin{array}{ccccccl}
\Omega ''a_{n-m}^{n-k}&(j,j)&,&
\Omega 'a_{n-m}^n&(j,j+n-m)&,&j=1,\ldots ,n-m-1~,\\ \\ 
(n-m)\Omega ''a_{n-m}^{n-k}&(\nu +n-m-1,\nu )&,&&&,&
\nu =1\ldots ,n-m ~.\end{array}$$
Part (1) follows from  
$\det S^*=
\pm (\Omega 'a_{n-m}^n)^{n-m-1}((n-m)\Omega ''a_{n-m}^{n-k})^{n-m}\not\equiv 0$.
\end{proof}


\begin{proof}[Proof of part (2):]
We prove that for $a_i=0$, $k\neq i\neq n$, the polynomial $\tilde{D}_m$ 
is of the form $\Omega ^{\dagger}a_n^{n-m}+\Omega ^{\dagger \dagger}a_n^{n-m-k}a_k^n$. 
Indeed, we list below the nonzero entries of the matrix $S(P,P_*)$ 
and their positions:

$$\begin{array}{ccccccccl}
1&(j,j)&,&\hspace{8mm}&a_k&(j,j+k)&\hspace{8mm}&,&\\ \\ 
a_n&(j,j+n)&&&&&&,&j=1,\ldots ,n-m~,\\ \\ 
b_0&(\nu +n-m,\nu )&,&&b_ka_k&
(\nu +n-m,\nu +k)&&,&
\nu =1,\ldots ,n~.\end{array}$$ 
One can develop $n-m-k$ times $\det S(P,P_*)$ w.r.t. its last column, 
where the only nonzero entries equal $a_n$. Thus 
$\det S(P,P_*)=\pm a_n^{n-m-k}\det H$, where $H$ is obtained from 
$S(P,P_*)$ by deleting the last $n-m-k$ columns and the rows with indices 
$k+1$, $\ldots$, $n-m$. The matrix $H$ has the following nonzero entries, in 
the following positions:

$$\begin{array}{ccccccccl}
1&(j,j)&,&\hspace{8mm}&a_k&(j,j+k)&&,&\\ \\
a_n&(j,j+n)&&&&&\hspace{8mm}&,&j=1,\ldots ,k~,\\ \\ 
b_0&(\nu +k,\nu )&,&&b_ka_k&
(\nu +k,\nu +k)&&,&
\nu =1,\ldots ,n~.\end{array}$$
For $\mu =1$, $\ldots$, $k$ one can subtract the $(\mu +k)$th row multiplied 
by $1/b_k$ from the $\mu$th one to make disappear the terms $a_k$ in 
positions $(\mu ,\mu +k)$; the entries $1$ in positions $(\mu ,\mu )$ change 
to $\Omega ^*:=1-b_0/b_k$. We denote the newly obtained matrix by $H_1$. 
Obviously 
$\det H_1=\det H$; we list the nonzero entries of $H_1$ and their respective 
positions:

$$\begin{array}{ccccccccl}
\Omega ^*&(j,j)&,&\hspace{8mm}&a_n&(j,j+n)&\hspace{8mm}&,&j=1,\ldots ,k~,\\ \\  
b_0&(\nu +k,\nu )&,&&b_ka_k&
(\nu +k,\nu +k)&&,&
\nu =1,\ldots ,n~.\end{array}$$
One applies Lemma~\ref{lm2diag} with $p=n+k$, $s=n$ 
to the matrix $H_1$ to conclude that 
$\det H=\det H_1=(\Omega ^*)^k(b_ka_k)^n\pm b_0^na_n^k$, so 
$\tilde{D}_m=\det S(P,P_*)=
\Omega ^{\dagger}a_n^{n-m}+\Omega ^{\dagger \dagger}a_n^{n-m-k}a_k^n$. 

But then the $(2n-1)\times (2n-1)$-Sylvester matrix 
$S(\tilde{D},\partial \tilde{D}/\partial a_k,a_k)$ has only the following 
nonzero entries, in the following positions:

$$\begin{array}{ccccccccl}
\Omega ^{\dagger \dagger}a_n^{n-m-k}&(j,j)&,&\hspace{8mm}&
\Omega ^{\dagger}a_n^{n-m}&(j,j+n)&\hspace{8mm}&,&j=1,\ldots ,n-1~,\\ \\ 
n\Omega ^{\dagger \dagger}a_n^{n-m-k}&(\nu +n-1,\nu )&&&&&&,
&\nu =1,\ldots ,n~.\end{array}$$
Its determinant equals 
$\pm (\Omega ^{\dagger}a_n^{n-m})^{n-1}(n\Omega ^{\dagger \dagger}a_n^{n-m-k})^n
\not\equiv 0$ which proves part (2).
\end{proof}

\begin{proof}[Proof of part (3):]
For $k=1$, $\ldots$ ,$n-m$ the polynomial 
$\tilde{D}_m$ contains the monomial 
$M_k:=\pm b_k^na_k^n(1-b_0/b_k)^ka_n^{n-m-k}$, and it  
does not contain any other monomial of 
the form $\Omega a_k^nE$, where $E$ is a product of powers 
of variables $a_i$ with $i\neq k$, see Proposition~\ref{propmonom}.  

Hence the first column of the $(2n-1)\times (2n-1)$-matrix 
$Y:=S(\tilde{D}_m,\partial \tilde{D}_m/\partial a_k,a_k)$ contains only two 
nonzero entries, and these are $Y_{1,1}=\pm b_k^n(1-b_0/b_k)^ka_n^{n-m-k}$ and 
$Y_{n,1}=\pm nb_k^n(1-b_0/b_k)^ka_n^{n-m-k}$. Thus 
$\Delta :=\det Y$  
is divisible by $a_n^{n-m-k}$. We consider two cases:

{\em Case 1:} $k=n-m$. We have to prove that $\tilde{D}_{m,n-m}|_{a_n=0}\not\equiv 0$. 
Set $a_j=0$ for $n-m\neq j\neq n-1$. Hence the nonzero entries of the matrix 
$S(P,P_*)$ and their positions are 

$$\begin{array}{ccccccl}
1&(j,j)&,&a_{n-m}&(j,j+n-m)&,&\\ \\ 
a_{n-1}&(j,j+n-1)&,&&&&j=1,\ldots ,n-m~,\\ \\ 
b_0&(\nu +n-m,\nu )&,&b_{n-m}a_{n-m}&(\nu +n-m,\nu +n-m)&,&\nu =1,\ldots ,n~.
\end{array}$$
One can subtract the $(j+n-m)$th row multiplied by $1/b_{n-m}$ from the 
$j$th one, $j=1$, $\ldots$, $n-m$, to make disappear the terms $a_{n-m}$ 
in the first $n-m$ rows. This doesn't change $\det S(P,P_*)$. The terms 
$1$ in positions $(j,j)$ are replaced by $1-b_0/b_{n-m}$. 
Hence $\tilde{D}_m$ is of the form  
$\Omega _1a_{n-m}^n+\Omega _2a_{n-1}^{n-m}a_{n-m}$ (one first develops 
$\det S(P,P_*)$ w.r.t. the last column, where there is a single nonzero 
entry $b_{n-m}a_{n-m}$ in position $(2n-m,2n-m)$, and then applies 
Lemma~\ref{lm2diag} with $p=2n-m-1$ and $s=n-1$). 

Thus the matrix 
$S^H:=S(\tilde{D}_m,\partial \tilde{D}_m/\partial a_{n-m},a_{n-m})$ 
contains only the following nonzero entries, in the following positions:

$$\begin{array}{ccccccl}
\Omega _1&(j,j)&,&\Omega _2a_{n-1}^{n-m}&(j,j+n-1)&,&j=1,\ldots ,n-1~,\\ \\ 
n\Omega _1&(\nu +n-1,\nu )&,&\Omega _2a_{n-1}^{n-m}&(\nu +n-1,\nu +n-1)&,&
\nu =1,\ldots ,n~.\end{array}$$
One can subtract the $(j+n-1)$st row from the $j$th one, $j=1$, $\ldots$, 
$n-1$, to make disappear the terms $\Omega _2a_{n-1}^{n-m}$ in the first 
$n-1$ rows; the terms $\Omega _1$ become $(1-n)\Omega _1$. Hence 
$\det S^H=\Omega _3a_{n-1}^{n(n-m)}\not\equiv 0$. 

{\em Case 2:} $1\leq k\leq n-m-1$. 
To prove that $\Delta$ is not divisible by $a_n^{n-m-k+1}$ 
we develop it w.r.t. 
its first column:

$$\Delta := (\pm b_k^n(1-b_0/b_k)^ka_n^{n-m-k})(\det ([Y]_{1,1})+
(-1)^{n+1}n\det ([Y]_{n,1}))~.$$
Our aim is to show that for $a_n=0$ the sum $Z:=\det ([Y]_{1,1})+
(-1)^{n+1}n\det ([Y]_{n,1})$ is nonzero; this implies $a_n^{n-m-k+1}$ not dividing 
$\Delta$. Notice that for $a_n=0$ the only 
nonzero entries in the second column of $Y$ (i.e. of 
$Y|_{a_n=0}=:Y^0$) are $Y^0_{1,2}$ and 
$Y^0_{n,2}=(n-1)Y^0_{1,2}$. Thus 

\begin{equation}\label{Z}
Z|_{a_n=0}=(Y^0_{1,2}+(-1)^{n+1}(-1)^nnY^0_{n,2})\det (Y^{\dagger})=
(1-n(n-1))Y^0_{1,2}(\det Y^{\dagger})~,
\end{equation}
where the matrix $Y^{\dagger}$ is obtained from $Y^0$ 
by deleting its first two columns, its first and its $n$th rows. 

\begin{lm}\label{lmY0}
The entry $Y^0_{1,2}$ is a not identically equal to $0$ polynomial in the 
variables $a_j$, $k\neq j\neq n$.
\end{lm} 

\begin{proof}
Indeed, this is the coefficient of 
$a_k^{n-1}$ in $R^0:=$Res$(P,P_*)|_{a_n=0}$. 
The matrix $S_*:=S(P,P_*)|_{a_n=0}$ has a 
single nonzero entry in its last column; 
this is $(S_*)_{2n-m,2n-m}=b_{n-m}a_{n-m}$. Hence $R^0=b_{n-m}a_{n-m}\det M$, where 
$M:=[S_*]_{2n-m,2n-m}$ ($M$ is $(2n-m-1)\times (2n-m-1)$). 

For $\nu =1,\ldots ,n-m$ one can subtract the $(n-m+\nu)$th row of $M$ 
multiplied by $1/b_k$ from its 
$\nu$th row to make disappear the terms $a_k$ in its first $n-m$ rows. The new 
matrix is denoted by $M^1$; one has $\det M=\det M^1$. 
The only terms of $\det M^1$ containing $a_k^{n-1}$ 
are now 
obtained by multiplying the entries $b_ka_k$ of the last $n-1$ rows of $M^1$. 
To get these terms up to a sign one has to multiply $(b_ka_k)^{n-1}$ by 
$\det M^*$, where $M^*$ is obtained from $M^1$ by deleting the rows and 
columns of the entries $b_ka_k$. The matrix $M^*$ is block-diagonal, its 
left upper block is upper-triangular and its right lower block is 
lower-triangular. The diagonal entries of these blocks (of sizes $k\times k$ 
and $(n-m-k)\times (n-m-k)$) equal $1-b_0/b_k$ and $a_{n-1}$. Hence 
$Y^0_{1,2}=\pm b_{n-m}a_{n-m}(1-b_0/b_k)^kb_k^{n-1}a_{n-1}^{n-m-k}\not\equiv 0$.   
\end{proof}

There remains to prove that $\det Y^{\dagger}\not\equiv 0$, see (\ref{Z}). 
The matrix $Y^{\dagger}$ is obtained as follows. 
Set $D^{\dagger}:=\tilde{D}_m|_{a_n=0}=\det S_*$; 
recall that $\det S_*=b_{n-m}a_{n-m}\det M^1$, see the proof of Lemma~\ref{lmY0}. 
Then $Y^{\dagger}=S(D^{\dagger},\partial D^{\dagger}/\partial a_k,a_k)$. Notice that 
$D^{\dagger}$ is a degree $n-1$, not $n$, polynomial in $a_k$, therefore 
$Y^{\dagger}$ is $(2n-3)\times (2n-3)$. It suffices to show that for 
$a_j=0$, $j\neq k$, $n-m$, $n-1$, one has $\det Y^{\dagger}\not\equiv 0$. This 
results from $\det M^1|_{a_j=0,k\neq j\neq n-1}$ not having multiple roots 
(which we prove below).
 
One can develop $n-m-k$ times $\det M^1$ w.r.t. its last column, 
where it has a single 
nonzero entry $a_{n-1}$, to obtain 
$\det M^1=\pm a_{n-1}^{n-m-k}\det M^{\dagger}$; $M^{\dagger}$ is 
$(n+k-1)\times (n+k-1)$, it is obtained from $M^1$ by 
deleting the last $n-m-k$ columns and the rows with indices 
$k+1$, $\ldots$, $n-m$. The matrix 
$M^{\dagger}$ satisfies the conditions of Lemma~\ref{lm2diag} with 
$p=n+k-1$ and $s=k$, the entries $r_j$ from the lemma equal $1-b_0/b_k\neq 0$ 
(for $j=1$, $\ldots$, $k$) or $b_ka_k$ (for $j=k+1$, $\ldots$, $n+k-1$); 
one has $q_j=a_{n-1}$ ($1\leq j\leq k$) or $q_j=b_0$ ($k+1\leq j\leq n+k-1$). 
Hence $\det M^{\dagger}|_{a_j=0,k\neq j\neq n-1}=
(1-b_0/b_k)^k(b_ka_k)^{n-1}\pm b_0^{n-1}a_{n-1}^k$. For 
$a_{n-1}\neq 0$ it has $n-1$ distinct roots. Part (3) is proved. 
\end{proof} 

\begin{proof}[Proof of part (4):]
We use sometimes the same notation as in the proof 
of part (3), but with different values of the indices, therefore the proofs 
of the two parts of the proposition should be considered as independent ones. 
For $k=n-m+1$, $\ldots$ ,$n$, the polynomial 
$\tilde{D}_m$ contains the monomial 
$N_{k-n+m}:=\pm b_{n-m}^{n-k}a_{n-m}^{n-k}b_0^ka_k^{n-m}$; it   
does not contain any other monomial of 
the form $\Omega a_k^{n-m}D$, where $D$ is a product of powers 
of variables $a_i$ with $i\neq k$, see Proposition~\ref{propmonom}. 

The first column of the $(2n-2m-1)\times (2n-2m-1)$-matrix 
$Y:=S(\tilde{D}_m,\partial \tilde{D}_m/\partial a_k,a_k)$ contains only two 
nonzero entries, namely $Y_{1,1}=\pm b_{n-m}^{n-k}a_{n-m}^{n-k}b_0^k$  
and $Y_{n-m,1}=\pm (n-m)b_{n-m}^{n-k}a_{n-m}^{n-k}b_0^k$. Thus 
$\det Y$ is divisible by $a_{n-m}^{n-k}$. We consider two cases:

{\em Case 1:} $k=n$. We show that $\det Y\not\equiv 0$ if $a_{n-m}=0$. We prove this for 
$a_j=0$, $n-m-1\neq j\neq n$. In this case the nonzero entries of the matrix 
$S(P,P_*)$ and their positions are

$$\begin{array}{ccccccl}
1&(j,j)&,&a_{n-m-1}&(j,j+n-m-1)&,&\\ \\ 
a_n&(j,j+n)&,&&&&j=1,\ldots, n-m~,\\ \\ 
b_0&(\nu +n-m,\nu )&,&b_{n-m-1}a_{n-m-1}&(\nu +n-m,\nu +n-m-1)&,&
\nu =1,\ldots ,n~.\end{array}$$
Subtracting the $(j+n-m)$th row multiplied by $1/b_{n-m-1}$ from the $j$th one 
for $j=1$, $\ldots$, $n-m$, one makes disappear the terms $a_{n-m-1}$ in the 
first $n-m$ rows. The only nonzero entry in the last column is now $a_n$ in 
position $(n-m,2n-m)$, so 

$$\det S(P,P_*)=(-1)^na_n\det [S(P,P_*)]_{n-m,2n-m}~.$$
The last matrix satisfies the conditions of Lemma~\ref{lm2diag} with 
$p=2n-m-1$, $s=n$ and one finds that its determinant is of the form 
$\Omega _4a_{n-m-1}^n+\Omega _5a_n^{n-m-1}$. Hence 
$\det S(P,P_*)=(-1)^na_n(\Omega _4a_{n-m-1}^n+\Omega _5a_n^{n-m-1})$. 
This means that the matrix 
$S(\tilde{D}_m,\partial \tilde{D}_m/\partial a_n,a_n)$ has only the 
following entries in the following positions:

$$\begin{array}{ccccll}
\Omega _5&(j,j)&,&\Omega _4a_{n-m-1}^n&(j,j+n-m-1)&,\\ \\ 
&&&&j=1,\ldots ,n-m-1&,\\ \\ 
(n-m)\Omega _5&(\nu +n-m-1,\nu )&,&\Omega _4a_{n-m-1}^n&
(\nu +n-m-1,\nu +n-m-1)&,\\ \\ &&&&\nu =1,\ldots ,n-m~.&\end{array}$$
One can subtract the $(j+n-m-1)$st row from the $j$th one, $j=1$, 
$\ldots$, $n-m-1$, to make disappear the terms $\Omega _4a_{n-m-1}^n$ in the 
first $n-m-1$ rows. The matrix becomes lower-triangular, with diagonal 
entries equal to $(1-n+m)\Omega _5$ or to $\Omega _4a_{n-m-1}^n$, so its 
determinant is not identically equal to $0$.  

{\em Case 2:} $n-m+1\leq k\leq n-1$. To prove that $\det Y$ is not 
divisible by $a_{n-m}^{n-k+1}$ 
we develop it w.r.t. 
its first column:

$$\det Y = (\pm b_{n-m}^{n-k}a_{n-m}^{n-k}b_0^k)(\det ([Y]_{1,1})+
(-1)^{n-m}(n-m)\det ([Y]_{n-m,1}))~.$$
Our aim is to show that for $a_{n-m}=0$ the sum $U:=\det ([Y]_{1,1})+
(-1)^{n-m}(n-m)\det ([Y]_{n-m,1})$ is nonzero; this implies 
$a_{n-m}^{n-k+1}$ not dividing 
$\det Y$. Notice that for $a_{n-m}=0$ the only 
nonzero entries in the second column of 
$Y^0:=Y|_{a_{n-m}=0}$ are $Y^0_{1,2}$ and 
$Y^0_{n-m,2}=(n-m-1)Y^0_{1,2}$. Thus 

\begin{equation}\label{U}
\begin{array}{ccl}
U|_{a_{n-m}=0}&=&(Y^0_{1,2}+(-1)^{n+1}(-1)^{n-m}(n-m)Y^0_{n-m,2})\det Y^{\dagger}\\ \\ 
&=&
(1-(n-m)(n-m-1))Y^0_{1,2}\det Y^{\dagger}~,\end{array}
\end{equation}
where the matrix $Y^{\dagger}$ is obtained from $Y^0$ 
by deleting its first two columns, its first and $(n-m)$th rows. 

\begin{lm}\label{lmY0bis}
The entry $Y^0_{1,2}$ is a not identically equal to $0$ polynomial in the 
variables $a_j$, $k\neq j\neq n-m$.
\end{lm}
\begin{proof} 
Indeed, this is the coefficient of 
$a_k^{n-m-1}$ in $R^0:=$Res$(P,P_*)|_{a_{n-m}=0}$. 
The matrix $S_*:=S(P,P_*)|_{a_{n-m}=0}$ has a 
single nonzero entry in its last column; 
this is $(S_*)_{n-m,2n-m}=a_n$. Hence $R^0=a_n\det M$, where 
$M:=[S_*]_{n-m,2n-m}$ ($M$ is $(2n-m-1)\times (2n-m-1)$). 
 
The only terms of $\det M$ containing $a_k^{n-m-1}$ 
are  
obtained by multiplying the entries $a_k$ of the first $n-m-1$ rows of $M$. 
To obtain these terms up to a sign one has to multiply $a_k^{n-m-1}$ by 
$\det M^*$, where $M^*$ is obtained from $M$ by deleting the rows and 
columns of the entries $a_k$. The matrix $M^*$ is block-diagonal, its 
left upper block is upper-triangular and its right lower block is 
lower-triangular. The diagonal entries of these blocks (of sizes $k\times k$ 
and $(n-m-k)\times (n-m-k)$) equal $b_0$ and $a_{n-m-1}$. Hence 
$Y^0_{1,2}=\pm a_nb_0^ka_{n-m-1}^{n-m-k}\not\equiv 0$.  
\end{proof}

There remains to prove that $\det Y^{\dagger}\not\equiv 0$, see (\ref{U}). 
The matrix $Y^{\dagger}$ is obtained as follows. 
Set $D^{\dagger}:=\tilde{D}_m|_{a_{n-m}=0}=\det S_*$; 
recall that $\det S_*=a_n\det M$ (see the proof of Lemma~\ref{lmY0bis}). 
Then $Y^{\dagger}=S(D^{\dagger},\partial D^{\dagger}/\partial a_k,a_k)$. Notice that 
$D^{\dagger}$ is a degree $n-m-1$, not $n-m$, polynomial in $a_k$, therefore 
$Y^{\dagger}$ is $(2n-2m-3)\times (2n-2m-3)$. It suffices to show that for 
$a_j=0$, $k\neq j\neq n-m-1$, one has $\det Y^{\dagger}\not\equiv 0$. This 
results from $\det M|_{a_j=0,k\neq j\neq n-m-1}$ not having multiple roots 
(which we prove below).

For $a_j=0$, $k\neq j\neq n-m-1$, one can develop $n-k$ times $\det M$
w.r.t. its last column in which there is a single nonzero entry 
$b_{n-m-1}a_{n-m-1}$ 
(on the diagonal). Hence $\det M=(b_{n-m-1}a_{n-m-1})^{n-k}\det M^{\dagger}$, where 
$M^{\dagger}$ is $(n-m+k-1)\times (n-m+k-1)$; it is 
obtained from $M$ by deleting the last $n-k$ rows and columns. 
The matrix $M^{\dagger}$ satisfies the conditions of Lemma~\ref{lm2diag} with 
$p=n-m+k-1$, $s=k$, $r_j=1$, $q_j=a_k$ ($j=1$, $\ldots$, $n-m-1$) or 
$r_j=a_{n-m-1}$, $q_j=b_0$ ($j=n-m$, $\ldots$, $n-m+k-1$). Hence 
$\det M^{\dagger}=a_{n-m-1}^k\pm b_0^ka_k^{n-m-1}$. For $a_{n-m-1}\neq 0$ it has 
$n-m-1$ distinct roots.

\end{proof}

\section{Some properties of the sets $\Theta$ and $\tilde{M}$
\protect\label{seclmst}}

\begin{lm}\label{lmsmooth}
Suppose that all roots of $P_*(.,a^0)$ ($a^0\in \mathcal{A}$) 
are simple and nonzero    
and that $P(.,a^0)$ and $P_*(.,a^0)$ have exactly one 
root in common. Then for any $j=n-m+1$, $\ldots$, $n$, 
in a neighbourhood of $a^0\in \mathcal{A}$ the set  
$\{ \tilde{D}_m=0\}$ is locally the graph of a smooth analytic 
function in the variables $a^j$. If in addition all roots of $P_{m,k}(.,a^0)$ 
are simple and nonzero ($1\leq k\leq n-m$), 
then in a neighbourhood of $a^0\in \mathcal{A}$ the set  
$\{ \tilde{D}_m=0\}$ is locally the graph of a smooth analytic 
function in the variables $a^k$.
\end{lm}

\begin{proof}
Denote by $[a]_{n-m}$ 
the first $n-m$ coordinates of $a\in \mathcal{A}$. 
Any simple root of $P_*$ is locally (in a 
neighbourhood of $[a^0]_{n-m}$) the value of a smooth analytic function 
$\lambda$ in the variables $[a]_{n-m}$. As $\lambda ([a^0]_{n-m})\neq 0$, 
the condition $P(\lambda ,a)/\lambda ^j=0$, $j<m$, 
allows to express $a_{n-j}$ locally (for $a_i$ close to $a_i^0$, $i\neq j$) 
as a smooth analytic function 
in the variables $a^{n-j}$. Suppose that all roots of $P_{m,k}(.,a^0)$ 
are simple and nonzero. Then any of these roots is a smooth analytic function 
in the variables $a^k$. This refers also to $\mu$, the root in common of 
$P$ and $P_*$ which is also a root of $P_{m,k}$. Hence one can express 
$a_k$ as a function in $a^k$ from the condition $P(\mu ,a)/\mu ^{n-k}=0$.
\end{proof}

\begin{st}\label{stMaxwell}
At a point of the Maxwell stratum the hypersurface $\{ \tilde{D}_m=0\}$ is 
locally the transversal intersection of two smooth analytic hypersurfaces along 
a smooth analytic subvariety of codimension $2$.
\end{st}

\begin{proof}
Suppose first that the roots in common of $P$ and $P_*$ are $0$ and $1$. 
The two conditions $P_*(0)=P_*(1)=0$ define a codimension $2$ 
linear subspace $\mathcal{S}$ in the space $\mathcal{A}$ 
of the variables $a$. Adding to them the two conditions 
$P(0)=P(1)=0$ means defining a codimension $2$ linear subspace 
$\mathcal{T}\subset \mathcal{S}$; hence $\mathcal{T}$ is a 
codimension $4$ linear subspace of $\mathcal{A}$. The two linear subspaces 
$\{ P(0)=0\}$ and $\{ P(1)=0\}$ and their intersections with 
$\{ P_*(0)=P_*(1)=0\}$ intersect transversally (along respectively 
$\{ P(0)=P(1)=0\}$ and $\mathcal{T}$). 

By means of a linear change $\tau ~:~x\mapsto \alpha x+\beta$, 
$\alpha \in \mathbb{C}^*$, $\beta \in \mathbb{C}$, one can transform any 
pair of distinct complex numbers into the pair $(0,1)$. Hence at a point of 
$\mathcal{T}$ the Maxwell stratum is locally the direct product of 
$\mathcal{T}$ and the two-dimensional orbit of the group of 
linear diffeomorphisms  
induced in the space $\mathcal{A}$ by the group of linear changes $\tau$. This 
proves the statement. 
\end{proof}

\begin{st}\label{stTheta}
(1) At a point of the set $\Theta$ (see Definition~\ref{ThetaM}) the set 
$\{ \tilde{D}_m=0\}$ is not representable as the graph in the space 
$\mathcal{A}$ of a smooth 
analytic function in the variables $a^j$, for any $j=n-m+1$, $\ldots$, $n$.

(2) At a point of the set $\Theta$ this set is a smooth analytic 
variety of dimension $n-2$ in the space of variables $a$.
\end{st}

\begin{proof}[Proof of part (1):] 
Suppose that for some $a=a_0\in \mathcal{A}$ one has 
$P_*(x_0,a_0)=P_*'(x_0,a_0)=0$. Suppose first that $x_0\neq 0$.  
Consider the equation 

\begin{equation}\label{vareps}
P_*(x,a_0)=\varepsilon ~,~~\, \, {\rm where}~~\, \,  
\varepsilon \in (\mathbb{C},0)~.
\end{equation} 
Its left-hand side equals 
$P_*''(x_0,a_0)(x-x_0)^2/2+o((x-x_0)^2)$ (with $P_*''(a_0,x_0)\neq 0$). 
Thus locally (for $x$ close to 
$x_0$) one has 

\begin{equation}\label{vareps1}
x-x_0=(2/P_*''(x_0,a_0))^{1/2}\varepsilon ^{1/2}+
o(\varepsilon ^{1/2})~.
\end{equation} 
In a neighbourhood of $a_0\in \mathcal{A}$ one can introduce new coordinates 
two of which are $x_0$ and $\varepsilon$. Indeed, one can write

$$\begin{array}{ccl}
(n-m-1)!P_*'/n!&=&
(x-x_0)(x^{n-m-2}+g_1x^{n-m-3}+\cdots +g_{n-m-2})\\ \\ 
&=&x^{n-m-1}+b_1^*a_1x^{n-m-2}+
\cdots +b_{n-m-1}^*a_{n-m-1}~,\end{array}$$
where $b_j^*=(n-j)!(n-m-1)!/(n-m-j-1)!n!$. 
Hence 

$$\begin{array}{rclcrclc}
b_1^*a_1&=&g_1-x_0&~~,~~&b_2^*a_2&=&g_2-x_0g_1&~,~\ldots ~,~\\ \\  
b_{n-m-2}^*a_{n-m-2}&=&g_{n-m-2}-x_0g_{n-m-3}&~~,~~&
b_{n-m-1}^*a_{n-m-1}&=&-x_0g_{n-m-2}&~~.~~\end{array}$$ 
The Jacobian 
matrix $\partial (a_1$, $\ldots$, $a_{n-m-1})/$ $\partial (x_0$, $g_1$, $\ldots$, 
$g_{n-m-2})$ is, up to multiplication of the columns by nonzero constants 
followed by transposition, the Sylvester matrix of the polynomials 
$x-x_0$ and $x^{n-m-2}+g_1x^{n-m-3}+\cdots +g_{n-m-2}$. Its determinant is nonzero 
because $x_0$ is not a root of the second of these polynomials. 

Thus in the space of the variables $(a_1$, $\ldots$, $a_{n-m-1})$ one can choose 
as coordinates $(x_0$, $g_1$, $\ldots$, $g_{n-m-2})$. The polynomial $P_*$ is 
a primitive of $P_*'$ and $(-\varepsilon )$ can be considered as the 
constant of integration, see (\ref{vareps}), therefore 
$(x_0$, $g_1$, $\ldots$, $g_{n-m-2}$, $\varepsilon )$ can be chosen as 
coordinates in the space of the variables $(a_1$, $\ldots$, $a_{n-m})$. Adding 
to them $(a_{n-m+1}$, $\ldots$, $a_n)$, one obtains local coordinates 
in the space $\mathcal{A}$. 

Hence the double root $\mu$ of $P_*$ is not an analytic, 
but a multivalued function of the local coordinates in $\mathcal{A}$, see 
(\ref{vareps1}). Consider the condition $P(\mu ,a)/\mu ^{n-j}=0$. One can 
express from it $a_j$ ($n-m+1\leq j\leq n$) 
as a linear combination of the variables $a^j$ with coefficients depending 
on $\mu$. This expression is of the form $A+\varepsilon ^{1/2}B$, where $A$ 
and $B$ ($B\not\equiv 0$) 
depend analytically on the local coordinates in $\mathcal{A}$. 
This proves the statement for $x_0\neq 0$. For $x_0=0$ the 
statement also holds true -- if for $x_0=0$ the set $\{ \tilde{D}_m=0\}$ 
is locally the graph of a holomorphic function in the variables $a^j$, then 
this must be the case for nearby values of $x_0$ as well which is false.  
Such values exist -- the change $x\mapsto x+\delta$, $\delta \in \mathbb{C}$, 
shifts simultaneously by $-\delta$ all roots of $P$ 
(hence of all its nonconstant derivatives as well).
\end{proof}

\begin{proof}[Proof of part (2):] 
Denote by $\xi$ the root of $P_*'$ which is 
also a root of $P_*$ and of $P$. Then $\xi$ is a smooth analytic function 
in the variables $a^{\dagger}:=(a_1$, $\ldots$, $a_{n-m-1})$. The condition 
$P_*(\xi ,a)=0$ allows to express $a_{n-m}$ as a smooth analytic function 
$\alpha$ in the variables $a^{\dagger}$. Set $a^*:=a|_{a_{n-m}=\alpha (a^{\dagger})}$. 
One can express $a_n$ as a 
smooth analytic function in the variables $a_j$, $n-m\neq j\neq n$, from the 
condition $P(\xi ,a^*)=0$. Thus locally $\Theta$ is the graph of a smooth 
analytic vector-function in the variables $a_j$, $n-m\neq j\neq n$, 
with two components.  
\end{proof}

\begin{st}\label{stirred}
For $2\leq m\leq n-2$ the polynomials $B_{m,k}$ and $C_{m,k}$ 
defined in Theorem~\ref{maintm} are irreducible.
\end{st}

\begin{proof}
Irreducibility of the factor $B_{m,k}$ is proved by analogy with 
Proposition~\ref{propmonom}. (For $n-m+1\leq k\leq n$ the analogy is complete 
because after the dilatations $a_j\mapsto a_j/b_j$, $j=1$, $\ldots$, $n-m$, 
the polynomial $P_*$ becomes $b_0P^*$, where $P^*$ is the polynomial 
$P$ defined for $n-m$ instead of $n$. For $1\leq k\leq n-m$ the coefficients 
of the polynomial $P_{m,k}$ are not $a_j$ (we set $a_0=1$), 
but $(b_k-b_j)a_j$, and one can perform similar dilatations. 
Only the variable $a_k$ is absent; this, however, is not an obstacle 
to the proof of irreducibility. The details are left for the reader.)

Irreducibility of the factors $C_{m,k}$ can be proved like this. Denote by 
$\xi$ and $\eta$ two of the roots of $P_*$. They are multivalued functions 
of the coefficients $a_1$, $\ldots$, $a_{n-m}$. The system of two equations 
$P(\xi ,a)=P(\eta ,a)=0$ allows to express for $\xi \neq \eta$ the 
coefficients $a_n$ and $a_{n-1}$ as functions of $a_1$, $\ldots$, $a_{n-2}$. 
These multivalued functions are defined over a Zariski dense open 
subset of the space of variables $(a_1$, $\ldots$, $a_{n-2})$ from 
which irreducibility of the set $\tilde{M}$ follows. Hence its projections 
in the hyperplanes $\mathcal{A}^k$ are also irreducible. 
\end{proof}

\begin{rem}\label{remm=1}
{\rm In the case $m=1$ one cannot prove in the same way as above that the 
polynomials $C_{1,k}$ are irreducible because the coefficient $a_{n-1}$ is 
in fact $a_{n-m}$.}
\end{rem}

\begin{rem}\label{remintermediate}
{\rm Proposition~\ref{Amkprop}, Lemma~\ref{lmsmooth}, 
Statements~\ref{stMaxwell}, 
\ref{stTheta} and \ref{stirred} 
allow to conclude that $\tilde{D}_{m,k}$ is of the form 
$A_{m,k}B_{m,k}^{s_{m,k}}C_{m,k}^{r_{m,k}}$, 
where $s_{m,k}$, $r_{m,k} \in \mathbb{N}$. Indeed, the form of the factor 
$A_{m,k}$ is justified by Proposition~\ref{Amkprop}. 
It follows from Lemma~\ref{lmsmooth} and its proof 
that for $A_{m,k}B_{m,k}C_{m,k}\neq 0$ the polynomials $\tilde{D}_m$ and 
$\partial \tilde{D}_m/\partial a_k$, when considered as polynomials in $a_k$, 
have no root in common. Hence a priori $\tilde{D}_{m,k}$ is of the form 
$A_{m,k}B_{m,k}^{s_{m,k}}C_{m,k}^{r_{m,k}}$, with 
$s_{m,k}$, $r_{m,k} \in \mathbb{N}\cup 0$ (implicitly we use the irreducibility 
of $B_{m,k}$ and $C_{m,k}$ here). 
Statements~\ref{stMaxwell} and \ref{stTheta} 
imply that one cannot have $s_{m,k}=0$ or $r_{m,k}=0$. To prove formula 
(\ref{formula}) now means to prove that $s_{m,k}=1$, $r_{m,k}=2$. This is 
performed in the next sections.}
\end{rem}

\section{The case $m=n-2$\protect\label{secmn-2}}

\begin{prop}\label{propmn-2}
For $m=n-2$, $n\geq 4$, one has $s_{m,k}=1$ and $r_{m,k}=2$.
\end{prop}

\begin{proof}[Proof for $3\leq k\leq n$.]
For $3\leq k\leq n$ the polynomial $\tilde{D}_{n-2}$ is a degree $2$ 
polynomial in $a_k$, see Proposition~\ref{propmonom}, so one can set 
$\tilde{D}_{n-2}:=Ua_k^2+Va_k+W$ and 
$\partial \tilde{D}_{n-2}/\partial a_k=2Ua_k+V$, 
where $U$, $V$, $W\in \mathbb{C}[a^k]$, $U\not\equiv 0$. Hence 
$S(\tilde{D}_{n-2}$,$\partial \tilde{D}_{n-2}/\partial a_k,a_k)=
\left( \begin{array}{ccc}U&V&W\\ 2U&V&0\\ 0&2U&V\end{array}\right)$ and 
$\tilde{D}_{n-2,k}=U(4UW-V^2)$. The second factor is up to a sign the 
discriminant of the quadratic polynomial (in the variable $a_k$) 
$Ua_k^2+Va_k+W$. Up to a sign, $U$ is 
the determinant of the matrix $S^L$ obtained from $S(P,P_*)$ by deleting its 
first two rows and the columns, where its entries $a_k$ are situated. 
Hence $U=\omega a_2^{n-k}$, $\omega \in \mathbb{C}^*$. Indeed, $S^L$ is 
block-diagonal, with diagonal blocks of sizes $k\times k$ (upper left) and 
$(n-k)\times (n-k)$ (lower right). They are respectively upper- and 
lower-triangular, with diagonal entries equal to $b_0$ and $b_2a_2$.

For $a_2=0$ the factor $4UW-V^2$ reduces to $-V^2\in \mathbb{C}[a]$. From the 
following lemma we deduce (after its proof) that the factor $C_{n-2,k}$ 
must be squared. 

\begin{lm}\label{lmsquared}
The polynomial $-V^2$ is a 
quadratic polynomial in the variables $a_i$, $i=3$, $\ldots$, $n$, with the square 
of at least one of them present in $-V^2$. For $k<n$ (resp. $k=n$) 
it contains the monomial $a_n^2(b_0)^{2k}(b_1a_1)^{2(n-k)}$ 
(resp. $a_{n-1}^2(b_0)^{2(n-1)}(b_1a_1)^2$).
\end{lm} 

\begin{proof}
Indeed, if $k<n$, then 
set $V^*:=V|_{a_j=0,j\neq 1, k, n}$ and $S^*:=S(P,P_*)|_{a_j=0,j\neq 1, k, n}$. 
There are two entries $a_k$ (resp. $a_1$ and $a_n$) 
in $S^*$, in positions $(1,k+1)$ and 
$(2,k+2)$ (resp. $(1,2)$, $(2,3)$, and $(1,n+1)$, $(2,n+2)$). The other 
nonzero entries of $S^*$ are $b_0$ (resp. $b_1a_1$) 
in positions $(\nu +2,\nu )$ (resp. $(\nu +2,\nu +1)$), $\nu =1$, $\ldots$, 
$n$. Thus 

$$V^*=\det V^{**}+\det V^{***}~~,~~{\rm where}~~V^{**}=[S^*]_{1,k+1}|_{a_k=0}~~,~~
V^{***}=[S^*]_{2,k+2}|_{a_k=0}~.$$
The matrices $V^{**}$ and $V^{***}$ are $(n+1)\times (n+1)$. Hence 

$$\det V^{**}=(-1)^na_n\det [V^{**}]_{1,n+1}~~,~~\det V^{***}=0$$
(because all entries in the last column of $V^{***}$ equal $0$). The matrix 
$[V^{**}]_{1,n+1}$ is block-diagonal, with diagonal blocks of sizes 
$k\times k$ (left upper, it is upper-triangular) and $(n-k)\times (n-k)$ 
(right lower, it is lower-triangular). Their 
diagonal entries equal respectively $b_0$ and $b_1a_1$. Thus 

$$V^*=\det V^{**}=(-1)^na_n(b_0)^k(b_1a_1)^{n-k}~.$$
Hence for $k<n$ the term $-V^2$ contains the monomial 
$a_n^2(b_0)^{2k}(b_1a_1)^{2(n-k)}$. 

For $k=n$ we set $a_j=0$, $j\neq 1$, $n-1$, $n$,  
$S^{\dagger}:=S(P,P_*)|_{a_j=0,j\neq 1, n-1, n}$ and  
$V^{\dagger}:=V|_{a_j=0,j\neq 1, n-1, n}$. Hence 

$$V^{\dagger}=\det V^{\dagger \dagger}+\det V^{\dagger \dagger \dagger }~~,~~{\rm where}~~
V^{\dagger \dagger }=[S^{\dagger }]_{1,n+1}|_{a_n=0}~~,~~
V^{\dagger \dagger \dagger }=[S^{\dagger }]_{2,n+2}|_{a_n=0}~.$$
One has $\det V^{\dagger \dagger }=0$ (all entries in the last column are $0$) and 
$V^{\dagger \dagger \dagger }$ has an entry $a_{n-1}$ in position $(1,n)$; no other 
entry of $V^{\dagger \dagger \dagger }$ depends on $a_{n-1}$. Hence 
$\det V^{\dagger \dagger \dagger }$ contains the monomial 
$(-1)^{n+1}a_{n-1}\det [V^{\dagger \dagger \dagger }]_{1,n}$. The matrix 
$[V^{\dagger \dagger \dagger }]_{1,n}$ is block-diagonal, with diagonal blocks of 
sizes $(n-1)\times (n-1)$ (upper left, it is upper-triangular, with 
diagonal entries equal to $b_0$) and 
$1\times 1$ (lower right, it equals $b_1a_1$). Hence $-V^2$ contains the 
monomial $a_{n-1}^2(b_0)^{2(n-1)}(b_1a_1)^2$.
\end{proof}

The factor $C_{n-2,k}$ is a linear function in the variables 
$a_3$, $\ldots$, $a_n$, with coefficients depending on $a_1$ and $a_2$. Indeed, 
set $P_*:=b_0(x-\alpha )(x-\beta )$, $0\neq \alpha \neq \beta \neq 0$. 
One can choose $(\alpha ,\beta )$ as 
coordinates in the space $(a_1,a_2)$. The polynomial $P$ is obtained from 
$P_*$ by rescaling of its coefficients followed by $(n-2)$-fold integration 
with constants of integration 
of the form $\eta _sa_s$, $\eta _s\in \mathbb{Q}^*$, $s=3$, $\ldots$, $n$. 
Consider the two conditions $P(\alpha ,a)/\alpha ^{n-k}=0$ and 
$P(\beta ,a)/\beta ^{n-k}=0$. Each of them is a linear form in the variables 
$a_3$, $\ldots$, $a_n$, with coefficients depending on $a_1$ and $a_2$; the 
one of $a_k$ equals $1$. 
The projection of the Maxwell stratum in the space of the variables $a^k$ is 
given by the condition 

\begin{equation}\label{eqAB}
\beta ^{n-k}P(\alpha ,a)-\alpha ^{n-k}P(\beta ,a)=0~.
\end{equation}  
Its left-hand side is a linear form in the variables $a_3$, $\ldots$, $a_{k-1}$, 
$a_{k+1}$, $\ldots$, $a_n$, with coefficients depending on $\alpha$ and $\beta$. 
The presence of the monomial $a_n^2(b_0)^{2k}(b_1a_1)^{2(n-k)}$ or 
$a_{n-1}^2(b_0)^{2(n-1)}(b_1a_1)^2$ in $\tilde{D}_{n-2,k}$ (see Lemma~\ref{lmsquared}) 
implies that the factor $C_{n-2,k}$ must be squared. 

There remains to prove that $s_{n-2,k}=1$, see Remark~\ref{remintermediate}. 
The left-hand side of equation (\ref{eqAB}) is divisible by $\alpha - \beta$. 
Represent this expression in the form $(\alpha - \beta )Q(\alpha ,\beta ,a)$. 
The polynomial $Q$ depends in fact on $\alpha +\beta =-b_1a_1/b_0$, 
$\alpha \beta =b_2a_2/b_0$ and $a$, 
hence this is a polynomial in $a$ (denoted by $K(a)$). 

Clearly $K$ depends 
linearly on the variables $a_3$, $\ldots$, $a_n$. On the other hand 
$K$ is quasi-homogeneous. Hence $K$ is irreducible. Indeed, should $K$ be 
the product of two factors, then one of the two (denoted by $Z$) 
should not depend on any 
of the variables $a_3$, $\ldots$, $a_n$, i.e. $Z$ should be a polynomial in 
$a_1$ and $a_2$. 

This polynomial should divide the coefficients of 
all variables $a_3$, $\ldots$, $a_n$ in $K$. But for $3\leq s\leq n$ 
the coefficient of $a_s$ in 
$K$ equals (see (\ref{eqAB})) 
$c_s:=(\beta ^{n-k}\alpha ^{n-s}-\alpha ^{n-k}\beta ^{n-s})/(\alpha - \beta )$. 
Hence $Z$ divides $c_s-\beta c_{s-1}=\alpha ^{n-s-1}\beta ^{n-k}$ 
for all $s\neq k$, and by symmetry $Z$ divides $\alpha ^{n-k}\beta ^{n-s-1}$ 
for all $s\neq k$. Hence $Z=1$ and the polynomial $C_{n-2,k}$ equals 
$(\beta ^{n-k}P(\alpha ,a)-\alpha ^{n-k}P(\beta ,a))/(\alpha -\beta )$. Its 
quasi-homogeneous weight (QHW) is $2n-k-1$ (notation: QHW$(C_{n-2,k})=2n-k-1$). 
Indeed, one has to consider QHW$(\alpha )$ and QHW$(\beta )$ to be equal to 
$1$ because $\alpha$ and $\beta$ are roots of 
$P_*$ and their QHW is the same as the one of the 
variable $x$. 

Obviously QHW$(\tilde{D}_{n-2,k})=2$QHW$(U)+$QHW$(W)$ because 
$\tilde{D}_{n-2,k}=U(4UW-V^2)$ and $\tilde{D}_{n-2,k}$ is quasi-homogeneous. 
As $U=\omega a_2^{n-k}$, one has QHW$(U)=2(n-k)$. The polynomial 
$\tilde{D}_{n-2,k}$ contains a monomial $\tilde{\omega}a_2^n$, 
$\tilde{\omega}\neq 0$ (see Proposition~\ref{propmonom}). This monomial 
is contained also in $W=\tilde{D}_{n-2}|_{a_k=0}$ hence 
QHW$(\tilde{D}_{n-2})=$QHW$(W)=2n$. Thus 

$${\rm QHW}(\tilde{D}_{n-2,k})=2{\rm QHW}(U)+{\rm QHW}(\tilde{D}_{n-2})
=6n-4k~.$$
On the other hand one knows already that a priori 
$\tilde{D}_{n-2,k}=A_{n-2,k}B_{n-2,k}^{s_{n-2,k}}C_{n-2,k}^2$, $s_{n-2,k}\in \mathbb{N}$, 
$A_{n-2,k}=a_2^{n-k}$. Hence 

$$\begin{array}{ccccc}
s_{n-2,k}{\rm QHW}(B_{n-2,k})&=&{\rm QHW}(\tilde{D}_{n-2,k})-{\rm QHW}(A_{n-2,k})-
2{\rm QHW}(C_{n-2,k})&&\\ &=&6n-4k-2(n-k)-2(2n-k-1)&=&2~,\end{array}$$
and as $B_{n-2,k}=b_1^2a_1^2-4b_0b_2a_2$, one has QHW$(B_{n-2,k})=2$, so $s_{n-2,k}=1$.
\end{proof}

\begin{proof}[Proof for $k=1$ and $k=2$.] 
In order to deal with the cases $k=1$ and $k=2$ we need to know the 
degrees and 
quasi-homogeneous weights of certain polynomials in the variables~$a$:

\begin{lm}\label{lmQHW}
(1) $\deg _{a_1}\tilde{D}_{n-2}=n$, $\deg _{a_2}\tilde{D}_{n-2}=n$; 

(2) {\rm QHW}$(\tilde{D}_{n-2})=2n$, 
{\rm QHW}$(\partial \tilde{D}_{n-2}/\partial a_1)=2n-1$, 
{\rm QHW}$(\partial \tilde{D}_{n-2}/\partial a_2)=2n-2$; 

(3) {\rm QHW}$(\tilde{D}_{n-2,1})=n(3n-2)$; 

(4) {\rm QHW}$(\tilde{D}_{n-2,2})=2n(n-1)$;

(5) {\rm QHW}$(B_{n-2,1})=n(n-1)$, {\rm QHW}$(B_{n-2,2})=n(n-1)$; 

(6) {\rm QHW}$(C_{n-2,1})=n(n-1)$, {\rm QHW}$(C_{n-2,2})=n(n-1)/2$.
\end{lm}

For $k=1$ or $2$ one has to find positive integers $u$ and $v$ such that 

$${\rm QHW}(\tilde{D}_{n-2,k})=(2-k)n+u{\rm QHW}(B_{n-2,k})+v{\rm QHW}(C_{n-2,k})~,$$
because $A_{n-2,k}=a_n^{2-k}$. For $k=2$ parts (4), (5) and (6) of the lemma 
imply that 
$u=1$, $v=2$ 
is the only possible choice. For $k=1$ there remain two possibilities -- 
$(u,v)=(1,2)$ or $(u,v)=(2,1)$ -- so we need another lemma as well:

\begin{lm}\label{lmelim}
For $a_j=0$, $j\neq 1$, $n-1$, $n$, the polynomials $\tilde{D}_{n-2}$, 
$\tilde{D}_{n-2,1}$, $B_{n-2,1}$ and $C_{n-2,1}$ are of the form respectively 
(with $\Delta _i\neq 0$)
$$\begin{array}{lccclccc}
\tilde{D}_{n-2}&=&\Delta _1a_na_1^n+\Delta _2a_na_{n-1}a_1+\Delta _3a_n^2&,& 
\tilde{D}_{n-2,1}&=&\Delta _4a_n^{2n-1}a_{n-1}^n+\Delta _5a_n^{3n-2}&,\\ \\  
B_{n-2,1}&=&\Delta _6a_n^{n-1}+\Delta _7a_{n-1}^n&{\rm and}&
C_{n-2,1}&=&\Delta _8a_n^{n-1}&.\end{array}$$
\end{lm}

The lemma implies that it is possible to have $(u,v)=(1,2)$, but not 
$(u,v)=(2,1)$. Indeed, otherwise the product 
$\tilde{D}_{n-2,1}=A_{n-2,1}B_{n-2,1}^2C_{n-2,1}$, with 
$A_{n-2,1}=a_n$, should contain three different monomials
whereas it contains only two.
\end{proof}

\begin{proof}[Proof of Lemma~\ref{lmQHW}.]
Parts (1) and (2) follow directly from Proposition~\ref{propmonom}. To prove 
parts (3) and (4) one has to observe that as the polynomial $\tilde{D}_{n-2}$ 
contains a monomial $c^*a_n^2$, $c^*\neq 0$, 
the $(2n-1)\times (2n-1)$-Sylvester matrices 
$S^*_k:=S(\tilde{D}_{n-2},\partial \tilde{D}_{n-2}/\partial a_k,a_k)$, 
$k=1$ or $2$, 
contain this monomial 
in positions $(j,j+n)$, $j=1$, $\ldots$, $n-1$ and only there. The matrix 
$S^*_1$ (resp. $S^*_2$) has entries $c^{\dagger}a_n$, $c^{\dagger}\neq 0$
(resp. $c^{**}\neq 0$) 
in positions 
$(\nu +n-1,\nu )$, $\nu =1$, $\ldots$, $n$. Hence 
$\tilde{D}_{n-2,k}$ contains a 
monomial $\pm (c^{\dagger}a_n)^n(c^*a_n^2)^{n-1}$ for $k=1$ and 
$\pm (c^{**})^n(c^*a_n^2)^{n-1}$ for $k=2$ whose quasi-homogeneous weight is 
respectively $n(3n-2)$ and $2n(n-1)$. 

To prove part (5) recall that the 
$(2n-1)\times (2n-1)$-Sylvester matrix 
$S^0:=S(P_{n-2,k},P_{n-2,k}')$, $k=1$ or $2$, 
has entries of the form $c^{**}a_n$, $c^{**}\neq 0$, in positions 
$(j,j+n)$, $j=1$, $\ldots$, $n-1$ and only there, and constant nonzero terms 
in positions $(\nu +n-1,\nu )$, $\nu =1$, $\ldots$, $n$. Thus $B_{n-2,k}$ 
contains a monomial $\pm c^{***}(a_n)^{n-1}$, $c^{***}\neq 0$ and  
QHW$(B_{n-2,k})=n(n-1)$. 
 
For the proof of part (6) we need to recall that the factors $C_{n-2,k}$ are 
related to polynomials $P$ divisible by $P_*$. When one performs this 
Euclidean division one obtains a rest of the form 
$U^{\dagger}(a)x+V^{\dagger}(a)$, where 
$U^{\dagger},V^{\dagger}\in \mathbb{C}[a]$, QHW$(U^{\dagger})=n-1$, 
QHW$(V^{\dagger})=n$, $U^{\dagger}$ (resp. $V^{\dagger}$) 
contains monomials 
$\omega _1a_1^{n-1}$ and $\omega _2a_{n-1}$ (resp. $\omega _3a_1^{n-2}a_2$ and 
$\omega _4a_n$), 
$\omega _i\neq 0$. (To see that the monomials $\omega _1a_1^{n-1}$ and 
$\omega _3a_1^{n-2}a_2$ are present one has to recall that at each step of the 
Euclidean division one replaces a term $Lx^s$, $L\in \mathbb{C}[a]$, 
by the sum 
$-L(b_1/b_0)a_1x^{s-1}-L(b_2/b_0)a_2x^{s-2}$.) 

To obtain the factor $C_{n-2,1}$ one has to eliminate $a_1$ from the system 
of equations $U^{\dagger}(a)=V^{\dagger}(a)=0$, i.e. one has to find the subset 
in the space of variables $a^1$ for which $U^{\dagger}$ and $V^{\dagger}$ 
have a common zero when 
considered as polynomials in $a_1$. The $(2n-3)\times (2n-3)$-Sylvester 
matrix $S(U^{\dagger},V^{\dagger},a_1)$ contains terms $\omega _2a_{n-1}$ 
in positions 
$(j,j+n-1)$, $j=1$, $\ldots$, $n-2$, and terms $\omega _3a_2$ in positions 
$(\nu +n-2,\nu )$, $\nu =1$, $\ldots$, $n-1$. Hence $C_{n-2,1}$ contains a 
monomial $\pm (\omega _2a_{n-1})^{n-2}(\omega _3a_2)^{n-1}$, of 
quasi-homogeneous weight $n(n-1)$. 

The proof of the second statement of part (6) is performed separately 
for the cases of even and odd $n$. If $n$ is even, then $U^{\dagger}$ 
(resp. $V^{\dagger}$) 
contains monomials $\Omega _1a_1a_2^{n/2-1}$ and $\Omega _2a_{n-1}$ 
(resp. $\Omega _3a_2^{n/2}$ and $\Omega _4a_n$), $\Omega _i\neq 0$. The 
$(n-1)\times (n-1)$-Sylvester matrix  $S(U^{\dagger},V^{\dagger},a_2)$ 
contains terms 
$\Omega _4a_n$ in positions 
$(j,j+n/2)$, $j=1$, $\ldots$, $n/2-1$, and $\Omega _1a_1$ in positions 
$(\nu +n/2-1,\nu )$, $\nu =1$, $\ldots$, $n/2$. Hence $C_{n-2,2}$ contains a 
monomial $\pm (\Omega _4a_n)^{n/2-1}(\Omega _1a_1)^{n/2}$, of 
quasi-homogeneous weight $n(n-1)/2$. 

When $n$ is odd, then $U^{\dagger}$ (resp. $V^{\dagger}$) 
contains monomials $\tilde{\Omega}_1a_2^{(n-1)/2}$ and $\tilde{\Omega}_2a_{n-1}$ 
(resp. $\tilde{\Omega}_3a_1a_2^{(n-1)/2}$ and $\tilde{\Omega}_4a_n$), 
$\tilde{\Omega}_i\neq 0$. The 
$(n-1)\times (n-1)$-Sylvester matrix  
$S(U^{\dagger},V^{\dagger},a_2)$ contains terms 
$\tilde{\Omega}_2a_{n-1}$ in positions 
$(j,j+(n-1)/2)$, $j=1$, $\ldots$, $(n-1)/2$, and 
$\tilde{\Omega}_3a_1$ in positions 
$(\nu +(n-1)/2,\nu )$, $\nu =1$, $\ldots$, $(n-1)/2$. Thus 
$C_{n-2,2}$ contains a monomial 
$\pm (\tilde{\Omega}_2a_{n-1}\tilde{\Omega}_3a_1)^{(n-1)/2}$, of 
quasi-homogeneous weight $n(n-1)/2$.

\end{proof}

\begin{proof}[Proof of Lemma~\ref{lmelim}.]
One can develop $\det S(P,P_*)$ w.r.t. the last column in which there 
is a single nonzero entry ($a_n$, in position $(2,n+2)$). Hence 
$\tilde{D}_{n-2}=(-1)^na_n\det S^{\sharp}$, where 
$S^{\sharp}:=[S(P,P_*)]_{2,n+2}$. The last column of $S^{\sharp}$ contains only 
two nonzero entries ($a_n$ in position $(1,n+1)$ and $b_1a_1$ 
in position $(n+1,n+1)$), therefore 

$$\det S^{\sharp}=(-1)^na_n\det S^{\sharp 1}+b_1a_1\det S^{\sharp 2}~~,~~
{\rm where}~~S^{\sharp 1}:=[S^{\sharp}]_{1,n+1}~~,~~
S^{\sharp 2}:=[S^{\sharp}]_{n+1,n+1}~.$$
The matrix $S^{\sharp 1}$ is upper-triangular, with diagonal 
entries equal to $b_0$, so $\det S^{\sharp 1}=b_0^n$, while $S^{\sharp 2}$ contains 
only two nonzero entries in its last column ($a_{n-1}$ in position $(1,n)$ and 
$b_1a_1$ in position $(n,n)$). Hence 

$$\det S^{\sharp 2}=(-1)^{n+1}a_{n-1}\det S^{\sharp 3}+b_1a_1\det S^{\sharp 4}
~~,~~{\rm where}~~S^{\sharp 3}:=[S^{\sharp 2}]_{1,n}~~,~~
S^{\sharp 4}:=[S^{\sharp 2}]_{n,n}~.$$
The matrix $S^{\sharp 3}$ is upper-triangular, with diagonal entries equal to 
$b_0$, so $\det S^{\sharp 3}=b_0^{n-1}$. The matrix $S^{\sharp 4}$ becomes 
lower-triangular after subtracting its second row multiplied by $1/b_1$ from 
the first one, 
with diagonal entries $1-b_0/b_1$, $b_1a_1$, $\ldots$, $b_1a_1$, from which 
the form of $\tilde{D}_{n-2}$ follows. 

Hence the $(2n-1)\times (2n-1)$-Sylvester matrix 
$S(\tilde{D}_{n-2},\partial \tilde{D}_{n-2}/\partial a_1,a_1)$ has only the 
following nonzero entries, in the following positions:

$$\begin{array}{rlllll}
\Delta _1a_n&(j,j)~,&\Delta _2a_na_{n-1}&(j,j+n-1)~,&
\Delta _3a_n^2&(j,j+n)~,\\ \\ &&&&&j=1,\ldots ,n-1,\\ \\ 
n\Delta _1a_n&(\nu +n-1,\nu )~,&\Delta _2a_na_{n-1}&(\nu +n-1,\nu 
+n-1)~,&&\nu =1,\ldots ,n~.
\end{array}$$
One can subtract the $(j+n-1)$st row from the $j$th one ($j=1,\ldots ,n-1$) 
to make disappear the terms $\Delta _2a_na_{n-1}$ in positions 
$(j,j+n-1)$. This does not change the determinant; the entries 
$\Delta _1a_n$ in positions $(j,j)$ become $(1-n)\Delta _1a_n$. The form of 
$\tilde{D}_{n-2,1}$ follows now from Lemma~\ref{lm2diag}.  

For $a_j=0$, $j\neq 1$, $n-1$, $n$, the polynomial $P_{n-2,1}$ is of the form 
$\alpha _1x^n+\alpha _2a_{n-1}x+\alpha _3a_n$, $\alpha _i\neq 0$, so 
the $(2n-1)\times (2n-1)$-Sylvester matrix $S(P_{n-2,1},P_{n-2,1}')$ 
has nonzero entries only

$$\begin{array}{rllllll}
\alpha _1&{\rm at}~(j,j)~,&\alpha _2a_{n-1}&{\rm at}~(j,j+n-1)~,&
\alpha _3a_n&{\rm at}~(j,j+n)~,&j=1,\ldots ,n-1,\\ \\ 
n\alpha _1&{\rm at}~(\nu ,\nu )~,&\alpha _2a_{n-1}&{\rm at}~(\nu ,\nu 
+n-1)~,&&&\nu =1,\ldots ,n~.
\end{array}$$
By analogy with the reasoning about $\tilde{D}_{n-2,1}$ one finds that 
$B_{n-2,1}=\Delta _6a_n^{n-1}+\Delta _7a_{n-1}^n$. 

To justify the form of $C_{n-2,1}$ it suffices to observe that for 
$a_j=0$, $j\neq 1$, $n-1$, $n$, one has (see the definition of 
$U^{\dagger}$ and $V^{\dagger}$ in the proof of Lemma~\ref{lmQHW}) 
$U^{\dagger}=\alpha _4a_{n-1}+\alpha _5a_1^{n-1}$, 
$V^{\dagger}=\alpha _6a_n$, $\alpha _i\neq 0$, so $\deg _{a_1}U^{\dagger}=n-1$ and   
$\deg _{a_1}V^{\dagger}=0$. When eliminating $a_1$ from the system of 
equalities $U^{\dagger}=V^{\dagger}=0$ one obtains Res$(U^{\dagger},V^{\dagger},a_1)=0$, 
i.e. $(\alpha _6a_n)^{n-1}=0$. 
\end{proof}

\section{The proof of $s_{m,1}=1$\protect\label{seck=1}}

In the present section we prove the following 

\begin{prop}\label{propk=1}
With the notation of Remark~\ref{remintermediate} one has $s_{m,1}=1$.
\end{prop}

The proof of the proposition makes use of the following lemma:

\begin{lm}\label{lmk=1}
Set $a_j=0$ for $j\neq 1$, $\ell$ and $n$, where $n-m+1\leq \ell \leq n-1$. 
Then $S(P,P_*)$ is of the form 
$\Omega _1a_n^{n-m-1}a_1^n+\Omega _2a_n^{n-m-1}a_{\ell}a_1^{n-\ell}+\Omega _3a_n^{n-m}$.
\end{lm}

\begin{proof}[Proof of Proposition~\ref{propk=1}:] 
Lemma~\ref{lmk=1} with $\ell =n-1$ implies that the matrix 
$S(\tilde{D}_m,\partial \tilde{D}_m/\partial a_1, a_1)$ has only the following 
nonzero entries, in the following positions:

$$\begin{array}{ccccl}
\Omega _1a_n^{n-m-1}&(j,j)&,&\Omega _2a_n^{n-m-1}a_{n-1}&(j,j+n-1)~~,\\ \\ 
\Omega _3a_n^{n-m}&(j,j+n)&,&&j=1,\ldots ,n-1~,\\ \\ 
n\Omega _1a_n^{n-m-1}&(\nu +n-1,\nu )&,&\Omega _2a_n^{n-m-1}a_{n-1}&
(\nu +n-1,\nu +n-1)~~,\\ \\ &&&&\nu =1,\ldots ,n~.
\end{array}$$
Subtract for $j=1$, $\ldots$, $n-1$ its $(j+n-1)$st row from the $j$th one. 
This preserves its determinant and leaves only the following nonzero entries, 
in the following positions:

$$\begin{array}{ccccll}
(1-n)\Omega _1a_n^{n-m-1}&(j,j)&,&\Omega _3a_n^{n-m}&(j,j+n)&,\\ \\ 
&&&&j=1,\ldots ,n-1&,\\ \\ 
n\Omega _1a_n^{n-m-1}&(\nu +n-1,\nu )&,&\Omega _2a_n^{n-m-1}a_{n-1}&
(\nu +n-1,\nu +n-1)&,\\ \\ &&&&\nu =1,\ldots ,n~.&
\end{array}$$ 
The new matrix satisfies the conditions of Lemma~\ref{lm2diag} with 
$p=2n-1$, $s=n$. Hence its determinant is of the form 

\begin{equation}\label{PPPP}
a_n^{(n-m-1)(2n-1)}(\Omega _4a_{n-1}^n+\Omega _5a_n^{n-1})~~,
\end{equation}
where $\Omega _4=((1-n)\Omega _1)^{n-1}\Omega _2^n$ and 
$\Omega _5=\pm \Omega _3^{n-1}\Omega _1^n$. The polynomial 
Res$(P_{m,1},P_{m,1}')$ contains monomials $\alpha a_n^{n-1}$ and $\beta a_{n-1}^n$, 
$\alpha \neq 0\neq \beta$; this can be proved by complete analogy with 
the analogous statement of Proposition~\ref{propmonom} with $m=1$ and we 
leave the proof for the reader. Hence the polynomial (\ref{PPPP}) is not 
divisible by a power of Res$(P_{m,1},P_{m,1}')$ higher than $1$, because in this 
case it would contain at least three different monomials in $a_n$ and 
$a_{n-1}$. Thus $s_{m,1}=1$.    
\end{proof}

\begin{proof}[Proof of Lemma~\ref{lmk=1}:]
The matrix $S(P,P_*)$ has only the following nonzero entries, 
in the following positions:

$$\begin{array}{cccccccl}
1&(j,j)&,&a_1&(j,j+1)&,&a_{\ell}&(j,j+\ell )~~,\\ \\ 
a_n&(j,j+n)&,& 
&&&&j=1,\ldots ,n-m~,\\ \\ 
b_0&(\nu +n-m,\nu )&,&b_1a_1&(\nu +n-m,\nu +1)&,&&\nu =1,\ldots ,n~.
\end{array}$$
One can develop the determinant $n-m-1$ times w.r.t. the last column in which 
each time there will be a single nonzero entry $a_n$. Thus 
$\det S(P,P_*)=\pm a_n^{n-m-1}\det S^{\ddagger}$, where the first row of 
$S^{\ddagger}$ contains the entries $1$, $a_1$, $a_{\ell}$ and $a_n$ in positions 
respectively $(1,1)$, $(1,2)$, $(1,\ell +1)$ and $(1,n+1)$; its second row 
is of the form $(b_0$, $b_1a_1$, $0$, $\ldots$, $0)$ and the next rows are 
the consecutive shifts of this one by one position to the right. Developing 
of $\det S^{\ddagger}$ w.r.t. the last column yields 

$$\det S^{\ddagger}=(-1)^na_n\det [S^{\ddagger}]_{1,n+1}+b_1a_1
\det [S^{\ddagger}]_{n+1,n+1}~.$$
The matrix $[S^{\ddagger}]_{1,n+1}$ is upper-triangular, with diagonal entries 
equal to $b_0$ (hence $\det [S^{\ddagger}]_{1,n+1}=b_0^n$). The determinant of the 
matrix $S^{\ddagger \ddagger}:=[S^{\ddagger}]_{n+1,n+1}$ can be developed 
$n-\ell -1$ times 
w.r.t. its last column, where each time 
it has a single nonzero entry $b_1a_1$ in its right lower corner:

$$\det S^{\ddagger \ddagger}=(b_1a_1)^{n-\ell -1}\det S^{*\dagger}~,$$
where $S^{*\dagger}$ is $(\ell +1)\times (\ell +1)$; it is obtained by 
deleting the last $n-\ell -1$ rows and columns of $S^{\ddagger \ddagger}$. The 
determinant $\det S^{*\dagger}$ can be developed w.r.t. its last column: 

$$\det S^{*\dagger}=(-1)^{\ell}a_{\ell}\det [S^{*\dagger}]_{1,\ell +1}+b_1a_1
\det [S^{*\dagger}]_{\ell +1,\ell +1}~.$$
The matrix $[S^{*\dagger}]_{1,\ell +1}$ (resp. $[S^{*\dagger}]_{\ell +1,\ell +1}$) 
is upper-triangular, with diagonal entries equal to $b_0$, so 
its determinant equals $b_0^{\ell}$ (resp. becomes lower-triangular (after 
subtracting its second row multiplied by $1/b_1$ from its first row), 
with diagonal 
entries equal to $1-b_0/b_1$, $b_1a_1$, $\ldots$, $b_1a_1$, so its determinant 
equals $(1-b_0/b_1)(b_1a_1)^{\ell -1}$). This implies the lemma. 
\end{proof}

\section{Completion of the proof of 
Theorem~\protect\ref{maintm}\protect\label{seccompletion}}

\begin{st}\label{stA} 
If formula (\ref{formula}) is true for $n=n_0$, $k=k_0$, 
then it is true for $n=n_0+1$, $k=k_0+1$.
\end{st}

\begin{st}\label{stAA}
If formula (\ref{formula}) is true for $n=n_0$, $m=n_0-2$, $k=1$, then 
it is true for $n=n_0$, $2\leq m<n_0-2$, $k=1$.
\end {st}

\begin{proof}[Proof of Statement~\ref{stA}:]
Recall that we have shown already 
(see Remark~\ref{remintermediate}) that for 
each $n$ fixed the polynomials $\tilde{D}_{m,k}$ ($2\leq m\leq n-2$, 
$1\leq k\leq n$) are of the form $A_{m,k}B_{m,k}^{s_{m,k}}C_{m,k}^{r_{m,k}}$, 
$s_{m,k}, r_{m,k} \in \mathbb{N}$. Suppose that for $4\leq n\leq n_0$ one has 
$s_{m,k}=1$, $r_{m,k}=2$. (Using MAPLE one can obtain this result  
for $n_0=4$.) Set $P(a,x):=x^{n_0}+a_1x^{n_0-1}+\cdots +a_{n_0}$, $a:=$
$(a_1,\ldots ,a_{n_0})$ and consider the polynomials $F:=ux^{n_0+1}+P$ and 
$F_*:=b_{-1}ux^{n_0-m+1}+P_*$, $u\in (\mathbb{C},0)$, $0\neq b_{-1}\neq b_j$ for 
$0\leq j\leq n_0-m$. They are deformations 
respectively of $P$ and $P_*$. Our reasoning uses the following 

\begin{obs}
{\rm One has} 

$$\begin{array}{lclc}F&=&u(x^{n_0+1}+x^{n_0}/u+\sum _{j=0}^{n_0-1}(a_{n-j}/u)x^j)&,\\ \\ 
F_*&=&u(b_{-1}x^{n_0-m+1}+b_0x^{n_0-m}/u+
\sum _{j=0}^{n_0-1}(b_{n-j}a_{n-j}/u)x^j)&,\end{array}$$
{\rm so after the change of parameters 
$\tilde{a}_1=1/u$, $\tilde{a}_s=a_{s-1}/u$, 
$s=2$, $\ldots$, $n_0$ (which is well-defined for $u\neq 0$) and the 
shifting by $1$ of the indices of the constants $b_j$, the polynomials $F$ and 
$F_*$ (up to multiplication by $1/u$) become $P$ and $P_*$ defined for $n_0+1$ 
instead of $n_0$.}
\end{obs}  

\begin{lm}\label{lmH}
The zero set of {\rm Res}$(F,F_*)$ for $u\neq 0$ 
is defined by 
an equation of the form $\tilde{D}_m+uH/d=0$, where 
$H\in \mathbb{C}[u,a]$ and $d\neq 0$. 
\end{lm}

\begin{proof}
Consider the 
$(2n_0-m+2)\times (2n_0-m+2)$-Sylvester matrix 
$\tilde{S}:=S(F,F_*)$. Permute the rows of $\tilde{S}$ as follows: 
place the $(n_0-m+2)$nd row in second position while shifting 
the ones with indices $2$, $\ldots$, $n_0-m+1$ by one position backward. 
This preserves up to a sign the determinant and yields a  
matrix $T$ which we decompose in four blocks the diagonal ones being of size 
$2\times 2$ (upper left, denoted by $T^*$) and 
$(2n_0-m)\times (2n_0-m)$ (lower right, denoted by $T^{**}$); the left lower 
block is denoted by $T^0$ and the right upper by $T^1$. 
An easy check shows that 

$$T^*=\left( \begin{array}{cc}u&1\\b_{-1}u&b_0
\end{array}\right) ~~,~~T^{**}|_{u=0}=S(P,P_*)~~,$$  
and that the only nonzero entries of the left lower block $T^0$ are 
$u$ and $b_{-1}u$, in positions 
$(3,2)$ and $(n_0-m+3,2)$ respectively.

Divide the first 
column of $T$ by $u$ (we denote the thus obtained matrix by $T^{\dagger}$). 
This does not change the zero set of $\det T$ for 
$u\neq 0$. For $u=0$ the matrix $T^{\dagger}$ is block-upper-triangular, 
with diagonal blocks equal to 
$\left( \begin{array}{cc}1&1\\b_{-1}&b_0\end{array}\right)$ and $S(P,P_*)$. 
Hence $\det T^{\dagger}=d\det S(P,P_*)+uH(u,a)$, 
$d:=\det T^*|_{u=0}=b_0-b_{-1}\neq 0$, 
$H\in \mathbb{C}[u,a]$. 
Thus the zero set of Res$(F,F_*)$ for $u\neq 0$ 
sufficiently small is defined by 
the equation $\tilde{D}_m+uH/d=0$.
\end{proof} 

For $u\neq 0$ (resp. for $u=0$) the quantity $\det T^{\dagger}$ 
is a degree $n_0-m+1$ (resp. $n_0-m$) polynomial in 
$a_k$ for $k=n_0-m+1$, $\ldots$, $n_0$, and a degree $n_0+1$ (resp. 
$n_0$) polynomial in $a_k$ for $k=1$, $\ldots$, $n_0-m$, see 
Proposition~\ref{propmonom}. Hence for each $k=1$, $\ldots$, $n_0$ 
there is one simple root $-1/w_k(u,a)$ of Res$(F,F_*)$ 
that tends to infinity as $u\rightarrow 0$. Thus one can set 
Res$(F,F_*)=(1+w_k(u,a)a_k)\tilde{D}^*_m$, where 
$\tilde{D}^*_m|_{u=0}\equiv \tilde{D}_m$ and 
$\deg _{a_k}\tilde{D}^*_m=n_0-m$ (resp. $n_0$) 
for $k=n_0-m+1$, $\ldots$, $n_0$ (resp. for $k=1$, $\ldots$, $n_0-m$). 

\begin{lm}\label{lmEm}
Set $E_m:=${\rm Res}$(F,F_*)$ and 
$\tilde{D}^*_{m,k}:=${\rm Res}$(E_m,\partial E_m/\partial a_k,
a_k)$. Then for $u\neq 0$ one has   
$\tilde{D}^*_{m,k}=\Omega ^{\flat \flat}(a_{n_0}^{2(n_0-m-k)}\tilde{D}_{m,k}
+uH_{m,k}(u,a))$, where $H_{m,k}\in \mathbb{C}[u,a]$.
\end{lm}

\begin{rem}\label{remuv}
{\rm One can set $u:=a_{n_0}^{2(n_0-m-k)}v$ to obtain the equality} 
$$\tilde{D}^*_{m,k}=\Omega ^{\flat \flat}a_{n_0}^{2(n_0-m-k)}(\tilde{D}_{m,k}
+vH_{m,k}(a_{n_0}^{2(n_0-m-k)}v,a))~.$$ 
{\rm Now in a neighbourhood of each $a_{n_0}\neq 0$ 
fixed the zero set of $\tilde{D}^*_{m,k}$ is defined by the equation 
$\tilde{D}_{m,k}+vH_{m,k}(a_{n_0}^{2(n_0-m-k)}v,a)=0$, i.e. by deforming the equation 
$\tilde{D}_{m,k}=0$.}
\end{rem}

\begin{proof}[Proof of Lemma~\ref{lmEm}:] 
Indeed, Proposition~\ref{propmonom} implies that $\tilde{D}_m$ 
contains a monomial $\Omega ^{\flat}a_k^{n_0}a_{n_0}^{n_0-m-k}$, $1\leq k\leq n_0-m$ 
(resp. $\Omega _{\flat}a_k^{n_0-m}a_{n_0-m}^{n_0-k}$, $n_0-m+1\leq k\leq n_0$) 
and this is the only 
monomial containing $a_k^{n_0}$ (resp. $a_k^{n_0-m}$). Similarly, $E_m$ contains 
a monomial $I:=u^{k+1}\Omega ^{\natural}a_k^{n_0+1}a_{n_0}^{n_0-m-k}$, $1\leq k\leq n_0-m$ 
(resp. $J:=u^{k+1}\Omega _{\natural}a_k^{n_0-m+1}a_{n_0-m}^{n_0-k}$, 
$n_0-m+1\leq k\leq n_0$) 
and this is the only monomial containing $a_k^{n_0+1}$ (resp. $a_k^{n_0-m+1}$). 
(The monomial $I$ is obtained as follows: one subtracts for $\nu =1$, $\ldots$, $n_0-m+1$ 
the $(\nu +n_0-m+1)$st row multiplied by $1/b_k$ from the $\nu$th one to make disappear 
the terms $a_k$ in the first $n_0-m+1$ rows. The monomial $I$ is the product of the terms 
$b_ka_k$ in the last $n_0+1$ rows, the terms $(1-1/b_k)u$ in the first $k+1$ rows and the terms 
$a_{n_0}$ in the next $n_0-m-k$ rows. The monomial $J$ is obtained in a similar way. One has 
to assume that QHW$(u)=-1$.) 
Knowing that $\deg _{a_k}E_m=n_0+1$ (resp. $\deg _{a_k}E_m=n_0-m+1$) for $u\neq 0$ 
and that $\deg _{a_k}\tilde{D}_m=n_0$ (resp. $\deg _{a_k}\tilde{D}_m=n_0-m$) 
one concludes that 

$$\begin{array}{cccccccl}
&E_m&=&u^{k+1}\Omega ^{\natural}a_k^{n_0+1}a_{n_0}^{n_0-m-k}&+&
\Omega ^{\flat}a_k^{n_0}a_{n_0}^{n_0-m-k}&+&uE^*(u,a)\\ \\ 
{\rm (resp.}
&E_m&=&u^{k+1}\Omega _{\natural}a_k^{n_0-m+1}a_{n_0-m}^{n_0-k}&+&
\Omega _{\flat}a_k^{n_0-m}a_{n_0-m}^{n_0-k}&+&uE^{**}(u,a)~~{\rm )~,}\end{array}$$
where $E^*$, $E^{**}\in \mathbb{C}[u,a]$, $\deg _{a_k}E^*\leq n_0$, 
$\deg _{a_k}E^{**}\leq n_0-m$. The Sylvester matrix 
$S(E_m$, $\partial E_m/\partial a_k,a_k)$ is $(2n_0+1)\times (2n_0+1)$ (resp. 
$(2n_0-2m+1)\times (2n_0-2m+1)$). We permute its rows by placing the 
$(n_0+1)$st (resp. $(n_0-m+1)$st) row in second position while shifting by 
one position backward the second, third, $\ldots$, $n_0$th (resp. $(n_0-m)$th) 
rows. The new matrix $T^{\flat}$ can be block-decomposed, with diagonal blocks 
$T^{u\ell}$ ($2\times 2$, upper left) and $T^{\ell r}$; the other two blocks are 
denoted by $T^{ur}$ and $T^{\ell \ell}$. Hence

$$\begin{array}{ccccl}
&T^{u\ell}&=&\left( \begin{array}{cc}
u^{k+1}\Omega ^{\natural}a_{n_0}^{n_0-m-k}&
\Omega ^{\flat}a_{n_0}^{n_0-m-k}+uX^1(u,a)\\ \\  
(n_0+1)u^{k+1}\Omega ^{\natural}a_{n_0}^{n_0-m-k}&
n_0\Omega ^{\flat}a_{n_0}^{n_0-m-k}+uX^2(u,a)\end{array}\right)&,\\ \\ 
{\rm (resp.}&
T^{u\ell}&=&\left( \begin{array}{cc}
u^{k+1}\Omega _{\natural}a_{n_0}^{n_0-m-k}&
\Omega _{\flat}a_{n_0}^{n_0-m-k}+uX^3(u,a)\\ \\ 
(n_0-m+1)u^{k+1}\Omega _{\natural}a_{n_0}^{n_0-m-k}&
(n_0-m)\Omega _{\flat}a_{n_0}^{n_0-m-k}+uX^4(u,a)\end{array}\right) &{\rm )~,}
\end{array}$$
$X^i\in \mathbb{C}[u,a]$. One has 
$T^{\ell r}|_{u=0}=S(\tilde{D}_m,\partial \tilde{D}_m/\partial a_k,a_k)$. The block 
$T^{\ell \ell}$ has just two nonzero entries, in its second column, and 
$T^{\ell \ell}|_{u=0}=0$. The first of these entries is in position 
$(3,2)$ and equals $u^{k+1}\Omega ^{\natural}a_{n_0}^{n_0-m-k}$ (resp. 
$u^{k+1}\Omega _{\natural}a_{n_0}^{n_0-m-k}$). The second of them is in position 
$(n_0+2,2)$ (resp. $(n_0-m+2,2)$) and equals 
$(n_0+1)u^{k+1}\Omega ^{\natural}a_{n_0}^{n_0-m-k}$ (resp. 
$(n_0-m+1)u^{k+1}\Omega _{\natural}a_{n_0}^{n_0-m-k}$). 

Thus for $u=0\neq a_{n_0}$ the zero set of $\tilde{D}_{m,k}^*$ is the one of 
$\tilde{D}_{m,k}$. For $u\neq 0$ small enough this set does not change if one 
divides the first column of the matrix $T^{\flat}$ by $u^{k+1}$. We denote the new 
matrix by $T^{\flat *}$. Obviously 
$\det T^{\flat *}=-\Omega ^{\natural}\Omega ^{\flat}(a_{n_0}^{2(n_0-m-k)}
\tilde{D}_{m,k}+uH_{m,k})$ (resp. 
$\det T^{\flat *}=-\Omega _{\natural}\Omega _{\flat}(a_{n_0}^{2(n_0-m-k)}
\tilde{D}_{m,k}+uH_{m,k})$) for a suitably defined polynomial $H_{m,k}$ which 
proves the lemma.
\end{proof}

Further to distinguish between the sets $\Theta$ and $\tilde{M}$ 
(see Definition~\ref{ThetaM}) 
defined for the polynomials $P$ or $F$ we write $\Theta _P$ and $\tilde{M}_P$ 
or $\Theta _F$ and $\tilde{M}_F$. 
Consider a point $A\in \Theta _P$ and a germ 
$\mathcal{G}$ of an affine space of dimension $2$ which intersects $\Theta _P$ 
transversally at $A$. Hence there exists a compact 
neighbourhood $\mathcal{N}$ of 
$A$ in the space $\mathcal{A}$ such that the parallel translates of 
$\mathcal{G}$ which intersect $\Theta _P$ at points of $\mathcal{N}$, 
intersect $\Theta _P$ transversally at these points. We assume that the value 
of $|a_{n_0}|$ remains $\geq \rho$ in $\mathcal{N}$ for some $\rho >0$. 
The restrictions of $\tilde{D}_{m,k}$ to each of 
these translates are smooth analytic functions each of which has one 
simple zero at its 
intersection point with $\Theta _P$; this follows from the factor $B_{m,k}$ 
participating in power $1$ in formula (\ref{formula}) for $n=n_0$. 
Hence for all  
$u\in \mathbb{C}$ with $0<|u|\ll \rho$ the restriction of $\tilde{D}^*_{m,k}$ 
to these translates are smooth analytic functions having simple zeros at the 
intersection points of the translates with $\Theta _P$. 

But this means that the power of the factor $B_{m,k}$ in formula (\ref{formula}) 
applied to the polynomial $F$ is equal to $1$ on the intersection of 
$\Theta _F$ with some open ball of dimension $n_0+1$ centered at $(0,A)$ in 
the space of the variables $(u,a)$. 
Hence this power equals $1$ on some Zariski open dense subset $\Theta ^0$ 
of $\Theta _F$ 
(if its complement $\Theta _F\backslash \Theta ^0$ is nonempty, then on 
$\Theta ^0$ this power might be $>1$). 
Thus the equality $s_{m,k}=1$ is justified for $n=n_0+1$, $2\leq k\leq n_0+1$ 
(because it is the coefficient of $x^{n_0-k}$, not of $x^{n_0+1-k}$ of $F$,  
that equals $a_k$). 

Now we adapt the above reasoning to the situation, where instead of a point 
$A \in \Theta _P$ one considers a point $A\in \tilde{M}_P$. Each of the 
translates of $\mathcal{G}$ intersects $\tilde{M}_P$ transversally, at just 
one point. The restriction of $\tilde{D}_{m,k}$ to the translate is a 
smooth analytic function having a double zero, so a priori the restriction of 
$\tilde{D}^*_{m,k}$ to it has either one double or two simple zeros. (Under an  
analytic deformation a double zero either remains such or splits into two 
simple zeros.) 
However two simple zeros is impossible because these zeros would be two 
points of $\tilde{M}_P$ whereas the translate contains just one point. Thus 
the power $2$ of the factor $C_{m,k}$ is justified for some 
Zariski open dense subset of $\tilde{M}_F$. Once again, 
this is sufficient to claim that 
formula (\ref{formula}) is valid for $n=n_0+1$ and for $2\leq k\leq n_0+1$. 

\end{proof}

\begin{proof}[Proof of Statement~\ref{stAA}:]
Recall that by Remark~\ref{remintermediate} we have to show that for $n=n_0$ 
one has $s_{m,k}=1$, $r_{m,k}=2$. The first of these equalities was proved in 
Section~\ref{seck=1} (see Proposition~\ref{propk=1}), so there remains to prove 
the second one. 

As in the proof of Statement~\ref{stA} we set 
$P(a,x):=x^{n_0}+a_1x^{n_0-1}+\cdots +a_{n_0}$, $a:=$
$(a_1,\ldots ,a_{n_0})$. We define the polynomial $P_*:=x^2+b_1a_1x+b_2a_2$ to 
correspond to the case $m=n_0-2$ 
(i.e. $b_k\neq 0$, $1$, $b_{3-k}$ for $k=1$, $2$). 
For $m=n_0-2$ Theorem~\ref{maintm} 
is proved in Section~\ref{secmn-2}, so we assume that $m<n_0-2$ and we set 
$G:=x^{n_0-m-2}P_*+u(b_3a_3x^{n_0-m-3}+\cdots +b_{n_0-m}a_{n_0-m})$, where 
$u\in (\mathbb{C},0)$ and for $i,j\geq 3$, $i\neq j$, one has 
$0\neq b_i\neq b_j\neq 0$. 
Denote by $G^{\sharp}$ the $(2n_0-m)\times (2n_0-m)$-matrix 
$S(P,G)$.

\begin{lm}\label{lmuu}
One has $\det G^{\sharp}|_{u=0}=a_{n_0}^{n_0-m-2}\det S(P,P_*)=
a_{n_0}^{n_0-m-2}\tilde{D}_2$. 
Hence $\tilde{G}:=\det G^{\sharp}=a_{n_0}^{n_0-m-2}\tilde{D}_2+uH^{\sharp}(u,a)$, 
$H^{\sharp}\in \mathbb{C}[u,a]$. 
\end{lm}
 
\begin{proof}
All nonzero entries of the matrix $G^{\sharp}$ in the intersection of 
its last $n_0-m-2$ columns 
and rows are $0$ for $u=0$. One can develop $n_0-m-2$ times 
$\det G^{\sharp}|_{u=0}$ w.r.t. its last column; 
each time there is a single nonzero entry in it  
which equals $a_{n_0}$. The matrix obtained from $G^{\sharp}|_{u=0}$ by deleting 
its last $n_0-m-2$ columns and the rows with indices $m+2$, $\ldots$, $n_0-1$ 
is precisely $S(P,P_*)$. 
\end{proof}

One can observe that $\det G^{\sharp}$ and $\det G^{\sharp}|_{u=0}$ are 
both degree $n_0$ polynomials in $a_1$. Assume that $a_{n_0}$ belongs to a 
closed disk on which one has $|a_{n_0}|\geq \rho ^{\flat}>0$. Suppose that 
$|u|\ll \rho ^{\flat}$, so one can consider the quantity 
$\tilde{D}_2+(u/a_{n_0}^{n_0-m-2})H^{\sharp}(u,a)$
as a deformation of $\tilde{D}_2$. To this end we set $u:=a_{n_0}^{n_0-m-2}v$, 
$v\in (\mathbb{C},0)$, see Remark~\ref{remuv}. 
Now to prove Statement~\ref{stAA} one has just to repeat 
the reasoning from the last paragraph of the proof of Statement~\ref{stA}.

\end{proof}

\end{document}